\documentclass[11pt]{article}%
\usepackage{geometry}
\usepackage{amsmath}
\usepackage{amsfonts}
\usepackage{amssymb}
\usepackage{color}
\usepackage{graphicx}%
\providecommand{\U}[1]{\protect\rule{.1in}{.1in}}
\newtheorem{theorem}{Theorem}
\newtheorem{acknowledgement}[theorem]{Acknowledgement}

\newtheorem{definition}[theorem]{Definition}

\newtheorem{lemma}[theorem]{Lemma}

\newtheorem{proposition}[theorem]{Proposition}
\newtheorem{remark}[theorem]{Remark}

\newenvironment{proof}[1][Proof]{\noindent\textbf{#1.} }{\ \rule{0.5em}{0.5em}}
\definecolor{lightgrey}{gray}{0.5}
\begin{document}

\title{Uniform convergence of proliferating particles to the FKPP equation}
\author{Franco Flandoli\thanks{Dipartimento di Matematica, Universit\`{a} di Pisa,
Italy. E-mail: \textsl{flandoli@dma.unipi.it}}, Matti
Leimbach\thanks{Technische Universit\"{a}t Berlin, Germany. E-mail:
\textsl{mattileimbach@msn.com}} \ and Christian Olivera\thanks{Departamento de
Matem\'{a}tica, Universidade Estadual de Campinas, Brazil. E-mail:
\textsl{colivera@ime.unicamp.br}. }}
\maketitle

\begin{abstract}
In this paper we consider a system of Brownian particles with proliferation
whose rate depends on the empirical measure. The dependence is more local than
a mean field one and has been called moderate interaction by Oelschl\"{a}ger
\cite{Oel1}, \cite{Oel2}. We prove that the empirical process converges,
uniformly in the space variable, to the solution of the
Fisher-Kolmogorov-Petrowskii-Piskunov equation. We use a semigroup approach
which is new in the framework of these systems and is inspired by some
literature on stochastic partial differential equations.

\end{abstract}



\section{Introduction\label{Intro}}

Consider the so called Fisher-Kolmogorov-Petrowskii-Piskunov (FKPP) equation -
with all constants equal to 1, which is always possible by suitable rescalings%
\begin{equation}
\frac{\partial u}{\partial t}=\Delta u+u\left(  1-u\right)  ,\qquad
u|_{t=0}=u_{0}. \label{FKPP}%
\end{equation}
This is a paradigm of equations arising in biology and other fields. For
instance, in the mathematical description of cancer growth, although being too
simplified to capture several features of true tumors, it may serve to explore
mathematical features of diffusion and proliferation. In such applications, it
describes a density of cancer cells which diffuse and proliferate with
proliferation modulated by the density itself, such that, starting with an
initial density $0\leq u_{0}\leq1$, the growth due to proliferation cannot
exceed the threshold $1$. Having in mind this example, it is natural to expect
that this equation is the macroscopic limit of a system of microscopic
particles, like cancer cells, which are subject to proliferation. To be
biologically realistic, we have to require that the proliferation rate is not
uniform among particles but depends on the concentration of particles:
wherever particles are more concentrated, there is less space and more
competition for nutrients, which slows down proliferation. We prove a result
of convergence of such kind of proliferation particle systems - as described
in detail in section \ref{section microscopic} below - to the FKPP equation. A
key point of the microscopic model that should be known in advance, to
understand this introduction, is that the proliferation rate of particle
``$a$" (see below the meaning of this index) is given by the random
time-dependent rate
\begin{equation}
\lambda_{t}^{a,N}=\left(  1-\left(  \theta_{N}\ast S_{t}^{N}\right)  \left(
X_{t}^{a,N}\right)  \right)  ^{+}, \label{definition lambda}%
\end{equation}
where $N$ is the number of initial particles, $X_{t}^{a,N}$ is the position of
particle ``$a$", $S_{t}^{N}$ is the empirical measure, $\theta_{N}$ is a
family of smooth mollifiers - hence $\theta_{N}\ast S_{t}^{N}$ is a smoothed
version of the empirical density. Formula (\ref{definition lambda}) quantifies
the fact that proliferation is slower when the empirical measure is more
concentrated, and stops above a threshold. Since there is no reason why the
mollified empirical measure $\theta_{N}\ast S_{t}^{N}$ is smaller than one, we
have to cut with the positive part, in (\ref{definition lambda}). Hence,
initially the limit PDE will have the proliferation term $u\left(  1-u\right)
^{+}$, which is meaningful also for $u>1$, but by a uniqueness result, the
term reduces to $u\left(  1-u\right)  $ when $0\leq u_{0}\leq1$.

The final result is natural and expected but there is a technical difficulty
which, in our opinion, is not sufficiently clarified in the literature. The
proof of convergence of the particle system to the PDE relies on the tightness
of the empirical measure and a passage to the limit in the identity satisfied
by the empirical measure. This identity includes the nonlinear term%
\[
\left\langle \left(  1-\theta_{N}\ast S_{t}^{N}\right)  ^{+}S_{t}^{N}%
,\phi\right\rangle
\]
where $\phi$ is a smooth test function. Since $S_{t}^{N}$ converges only
weakly, it is required that $\theta_{N}\ast S_{t}^{N}$ \textit{converges
uniformly}, in the space variable, in order to pass to the limit. Maybe in
special cases one can perform special tricks but the question of uniform
convergence is a natural one in this problem and it is also of independent
interest, hence we investigate when it holds true.

Following the proposal of K. Oelschl\"{a}ger \cite{Oel1}, \cite{Oel2}, we
assume
\begin{equation}
\theta_{N}\left(  x\right)  =N^{\beta}\theta\left(  N^{\beta/d}x\right)
.\label{theta N Oel}%
\end{equation}
Here $\theta$ is a probability density with a Sobolev regularity
$W^{\alpha_{0},2}\left(  \mathbb{R}^{d}\right)  $ specified by the technical
assumption (\ref{assumption on theta}) below. Recall that the case $\beta=0$
is the mean field one (long range interaction), the case $\beta=1$ corresponds
to local (like nearest neighbor) interactions, while the case $0<\beta<1$
corresponds to an intermediate regime, called \textquotedblleft
moderate\textquotedblright\ by \cite{Oel1}. Our main result is that uniform
convergence of $\theta_{N}\ast S_{t}^{N}$ to $u$ holds under the condition
\[
\beta<\frac{1}{2}.
\]
In addition to our main result, Theorem \ref{Thm 1}, see also Appendix
\ref{appendix B}\ where we show that this condition arises with other proofs
of uniform convergence. We believe this condition is strict for the uniform
convergence.
A second motivation for the analysis of uniform convergence, besides the
problem of passage to the limit in the nonlinear term outlined above, is the
question whether a "front" of microscopic particles which moves due to
proliferation approximates the traveling waves of FKPP equation. Results in
this direction seem to be related to uniform convergence of mollified
empirical measure but they require also several other ingredients and go
beyond the scope of the present paper, hence they are not discussed here.

\subsection{Comparison with related problems and results}

First, let us clarify that the problem treated here is more correct and
difficult than a two-step approach which does not clarify the true relation
between the particle system and the PDE, although it gives a plausible
indication of the link. The two-step approach freezes first the parameter in
the mollifier, namely it treats particles proliferating with rate%
\[
\lambda_{t}^{a,N_{0},N}=\left(  1-\left(  \theta_{N_{0}}\ast S_{t}^{N_{0}%
,N}\right)  \left(  X_{t}^{a,N_{0},N}\right)  \right)  ^{+}%
\]
and proves that $S_{t}^{N_{0},N}$ weakly converges as $N\rightarrow\infty$, to
the solution $u_{N_{0}}$ of the following equation with non-local
proliferation%
\begin{equation}
\frac{\partial u_{N_{0}}}{\partial t}=\Delta u_{N_{0}}+u_{N_{0}}\left(
1-\theta_{N_{0}}\ast u_{N_{0}}\right)  ^{+}. \label{mollified FKPP}%
\end{equation}
The second step consists in proving that $u_{N_{0}}$ converges to the solution
$u$ of the FKPP equation. The link between the particle system $X_{t}%
^{a,N_{0},N}$ and the solution $u$ of the FKPP equation is only conjectured by
this approach. In principle the conjecture could be even wrong. Take a system
of particle interactions with short range couplings, where the two-steps
approach leads to the porous media equation with the non-linearity $\Delta
u^{2}$ (see \cite{Phi}). But a direct link between the particle system and the
limit PDE (the so called hydrodynamic limit problem)\ leads to a non-linearity
of the form $\Delta f\left(  u\right)  $ where $f\left(  u\right)  $ is not
necessarily $u^{2}$ (see \cite{Va}, \cite{Uc}). For a proof of the
\textit{mean field} result of convergence of $S_{t}^{N_{0},N}$ to $u_{N_{0}}$
as $N\rightarrow\infty$, see for instance \cite{ChMeleard}, \cite{FlaLeim}.
The issue of uniform convergence of $\theta_{N}\ast S_{t}^{N}$ to $u$ does not
arise and weak convergence of the measures $S_{t}^{N_{0},N}$ is sufficient.

Going back to the problem with the rates (\ref{definition lambda}), K.
Oelschl\"{a}ger papers \cite{Oel1}, \cite{Oel2} have been our main source of
inspiration. Our attempt in the present work is to clarify\ a result of
convergence in the case of diffusion and proliferation under assumptions
comparable to those of \cite{Oel1}, \cite{Oel2} but possibly with some
additional degree of generality and with a new proof.

We have extended the assumption $\beta<\frac{d}{\left(  d+1\right)  \left(
d+2\right)  }$ and removed the restriction $V=W\ast W$ of \cite{Oel2} and,
hopefully, we have given a modern proof which clarifies certain issues of the
tightness and the convergence problem. Concerning extensions of the range of
$\beta$, maybe there are other directions, as remarked in \cite{Oel2}, page
575; our specific extension is however motivated not only by the generality
but also by the property of uniform convergence (not proved in \cite{Oel2}),
which seems relevant in itself.

Other interesting works related to the problem of particle approximation of
FKPP equation are \cite{Meleard}, \cite{MelRoelly}, \cite{Met}, \cite{Nappo},
\cite{Ste} and \cite{Baker}, \cite{BoVe} from the more applied literature. For
the FKPP limit of discrete lattice systems, even the more difficult question
of the hydrodynamic limit has been solved, see \cite{DeMFerrLeb} with
completely local interaction, but the analogous problem for diffusions is more
difficult and has not been done.

To solve the problem of uniform convergence, we propose a new approach, by
semigroup theory. Traces of this approach can be found in \cite{Met} and
\cite{ChMeleard}, but have been used for other purposes. In the work
\cite{FlaLeim} it is remarked that uniform convergence can be obtained as a
by-product of energy inequalities and Sobolev convergence, under the
assumption $\beta<\frac{d}{d+2}$, but only in dimension $d=1$, where the
condition is more restrictive than $\beta<1/2$.

The approach extends to other models, in particular with interactions. With
the same technique, under appropriate assumptions on the convolution kernels
$\theta_{N}$ below, we may recover a result, under different assumptions, of
\cite{Oel1}, where the macroscopic PDE has the form%
\[
\frac{\partial u}{\partial t}=\Delta u-\operatorname{div}\left(  uF\left(
u\right)  \right)  +u\left(  1-u\right)  ,\qquad u|_{t=0}=u_{0}%
\]
and $F$ is a local nonlinear function, not a non-local operator as in mean
field theories.

Let us insist on the fact that our proliferation rate is natural from the
viewpoint of Biology. It is very different from the constant rate used in the
probabilistic formulae used by McKean and others to represent solutions of the
FKPP equations; these formula have several reasons of interest but do not have
a biological meaning - constant proliferation rate would lead to exponential
blow-up of the number of particles. Constant rates do not pose the
difficulties described above in taking the limit in the nonlinear term.
Approximation by finite systems of these representation formula therefore pose
different problems. For this and other directions, different from our one, see
\cite{McKean}, \cite{RegnierTalay} and references therein.

\subsection{The microscopic model\label{section microscopic}}

We consider a particle system on filtered probability space $\left(
\Omega,\mathcal{F},\mathcal{F}_{t},P\right)  $ with $N\in\mathbb{N}$ initial
particles. We label particles by $a\in\Lambda^{N}$, where
\[
\Lambda^{N}=\left\{  \left(  k,i_{1},...,i_{n}\right)  \colon i_{1}%
,...,i_{n}\in\left\{  1,2\right\}  ,k=1,...,N,n\in\mathbb{N}_{0}\right\}
\]
is the set of all particles. For a non-initial particle $a=\left(
k,i_{1},...,i_{n}\right)  $ we denote its parent particle by $(a,-)=\left(
k,i_{1},...,i_{n-1}\right)  $. Each particle has a lifetime, which is the
random time interval $I^{a,N}=[T_{0}^{a,N},T_{1}^{a,N})\subset\lbrack
0,\infty)$, where $T_{0}^{a,N},T_{1}^{a,N}$ are $\mathcal{F}_{t}$-stopping
times. We have $T_{0}^{a,N}=0$ for initial particles $a=(k)$, $k=1,\dots,N$
and $T_{0}^{a,N}=T_{1}^{(a,-),N}$ for other particles. The time $T_{1}^{a,N}$
at which a particle dies and splits into two (we call this a proliferation
event) is described more precisely below.\newline Particles are born at the
position their parent died, i.e. $X_{T_{0}^{a,N}}^{a,N}=X_{T_{1}^{\left(
a,-\right)  ,N}}^{\left(  a,-\right)  ,N}$ with the convention $X_{T_{1}%
^{a,N}}^{a,N}:=\lim_{t\uparrow T_{1}^{a,N}}X_{t}^{a,N}$. During its lifetime
the position of $a\in\Lambda^{N}$, $X_{t}^{a,N}\in\mathbb{R}^{d}$, is given
by
\begin{equation}
dX_{t}^{a,N}=\sqrt{2}dB_{t}^{a} \label{initial SDE}%
\end{equation}
where $B^{a}$ are independent Brownian motions in $\mathbb{R}^{d}$.\newline
Let $\Lambda_{t}^{N}$ denote the set of all particles alive at time $t$. We
define the empirical measure as
\[
S_{t}^{N}=\frac{1}{N}\sum_{a\in\Lambda_{t}^{N}}\delta_{X_{t}^{a,N}}.
\]
Take a family of standard Poisson processes $\left(  \mathcal{N}^{0,a}\right)
_{a\in\Lambda^{N}}$ which is independent of the Brownian motion and the
initial condition $X_{0}^{(k),N}$, $k=1,\dots,N$. The branching time
$T_{1}^{a,N}$ of particle $a\in\Lambda^{N}$ is the first (and only) jump time
of $\mathcal{N}_{t}^{a,N}:=\mathcal{N}_{\Lambda_{t}^{a,N}}^{0,a}$, where
$\Lambda_{t}^{a,N}=\int_{0}^{t}1_{s\in I^{a,N}}\lambda_{s}^{a,N}ds$ and the
random rate $\lambda_{t}^{a,N}$ is given by
\[
\lambda_{t}^{a,N}=\left(  1-\left(  \theta_{N}\ast S_{t}^{N}\right)  \left(
X_{t}^{a,N}\right)  \right)  ^{+}%
\]
where
\begin{equation}
\theta_{N}(x)=\epsilon_{N}^{-d}\theta\left(  \epsilon_{N}^{-1}x\right)
\label{theta N eps}%
\end{equation}
is a family of mollifiers with
\[
\epsilon_{N}=N^{-\frac{\beta}{d}}%
\]
namely we assume (\ref{theta N Oel}).

\subsection{Assumptions and main result\label{sect macroscopic limit}}

Throughout this paper we assume that
\begin{equation}
\beta\in(0,\frac{1}{2}) \label{assumption on eps beta}%
\end{equation}
and that $\theta:\mathbb{R}^{d}\rightarrow\mathbb{R}$ is a probability density
of class%
\begin{equation}
\theta\in W^{\alpha_{0},2}\left(  \mathbb{R}^{d}\right)  \text{ for some
}\alpha_{0}\in\left(  \frac{d}{2},\frac{d(1-\beta)}{2\beta}\right]
\label{assumption on theta}%
\end{equation}
(notice that, for $\beta>0$, the inequality $\frac{d}{2}<\frac{d(1-\beta
)}{2\beta}$ is equivalent to $\beta<\frac{1}{2}$). The weaker assumption
$\beta=1$ corresponds to nearest-neighbor (or contact) interaction and it is
just the natural scaling to avoid that the kernel is more concentrated than
the typical space around a single particle, when the particles are uniformly
distributed. The case $\beta=0$ corresponds to mean field interaction. The
explanation for condition (\ref{assumption on eps beta}) is given at the
beginning of Section \ref{section main estimate}. At the biological level it
means that the modulation of proliferation by the local density of cells is
not completely local, but has a certain range of action, which is less than
long range as a mean field model.

Let us introduce the mollified empirical measure (the theoretical analog of
the numerical method of kernel smoothing) $h_{t}^{N} $ defined as%
\[
h_{t}^{N}(x)=\left(  \theta_{N}\ast S_{t}^{N}\right)  \left(  x\right)  .
\]

Concerning the initial condition, assume that $u_{0}\in L^{1}\left(
\mathbb{R}^{d}\right)  $, $0\leq u_{0}\left(  x\right)  \leq1$, $u_{0}$ is
uniformly continuous and $S_{0}^{N}$ converges weakly to $u_{0}\left(
x\right)  dx$, as $N\rightarrow\infty$, in probability. Moreover, assume that
for some $\rho_{0}\geq\alpha_{0}-1$
\begin{equation}
\sup_{N}E\left[  \int_{\mathbb{R}^{d}}\left\vert \left(  I-A\right)
^{\rho_{0}/2}h_{0}^{N}\left(  x\right)  \right\vert ^{2}dx\right]  <\infty.
\label{initial cond}%
\end{equation}
When the initial positions $X_{0}^{i}$, $i=1,...,N$, are independent
identically distributed with common probability density $u_{0}\in W^{\rho
_{0},2}\left(  \mathbb{R}^{d}\right)  $, with $\alpha_{0}-1\leq\rho_{0}%
\leq\alpha_{0}$, this condition is satisfied, see Proposition
\ref{Proposition sufficient cond} below. Finally, the definition of weak
solution of the PDE (\ref{FKPP}) is given below in Section
\ref{section auxiliary}.




\begin{theorem}
\label{Thm 1} \emph{Assume that $S_{0}^{N}$ converges weakly to $u_{0}\left(
x\right)  \mathrm{d}x$, as $N\rightarrow\infty$, in probability, where $u_{0}$
satisfies the assumptions above. Further, assume (\ref{assumption on eps beta}%
), (\ref{assumption on theta}) and (\ref{initial cond}). Then, for every
$\alpha\in(d/2,\alpha_{0})$, the process $h^{N}$ converges in probability in
the }

\begin{itemize}
\item \emph{weak star topology of $L^{\infty}\left(  0,T;L^{2}\left(
\mathbb{R}^{d}\right)  \right)  $, }

\item \emph{weak topology of $L^{2}\left(  0,T;W^{\alpha,2}\left(
\mathbb{R}^{d}\right)  \right)  $ }

\item \emph{strong topology of $L^{2}\left(  0,T;W_{loc}^{\alpha,2}\left(
\mathbb{R}^{d}\right)  \right)  $ }
\end{itemize}

\emph{as $N\rightarrow\infty$, to the unique weak solution of the PDE
(\ref{FKPP}). }
\end{theorem}

Note that the topology of convergences of $h_{t}^{N}$ includes the convergence
in $L^{2}\left(  0,T;C\left(  D\right)  \right)  $ for every regular bounded
domain $D\subset\mathbb{R}^{d}$. The notion of weak solution is given by
Definition \ref{def weak sol}.

\section{Preparation\label{section Formulation SPDE}}

\subsection{Analytic Semigroup and Sobolev Spaces\label{subsect Bessel}}

For every $\alpha\in\mathbb{R}$, the Sobolev spaces $W^{\alpha,2}\left(
\mathbb{R}^{d}\right)  $ are well defined, see \cite{Triebel} for the material
recalled here. For positive $\alpha$ the restriction of $f\in W^{\alpha
,2}\left(  \mathbb{R}^{d}\right)  $ to a ball $B\left(  0,R\right)  $ is in
$W^{\alpha,2}\left(  B\left(  0,R\right)  \right)  $. The family of operators,
for $t\geq0$,%
\[
\left(  e^{tA}f\right)  \left(  x\right)  =\int_{\mathbb{R}^{d}}\frac
{1}{\left(  4\pi t\right)  ^{d/2}}e^{-\frac{\left\vert x-y\right\vert ^{2}%
}{4t}}f\left(  y\right)  dy
\]
defines an analytic semigroup in each space $W^{\alpha,2}\left(
\mathbb{R}^{d}\right)  $. With little abuse of notation, we write $e^{tA}$ for
each value of $\alpha$. The infinitesimal generator, say in $L^{2}\left(
\mathbb{R}^{d}\right)  $, is the operator $A:D\left(  A\right)  \subset
L^{2}\left(  \mathbb{R}^{d}\right)  \rightarrow L^{2}\left(  \mathbb{R}%
^{d}\right)  $ defined as $Af=\Delta f$. Fractional powers $\left(
I-A\right)  ^{\beta}$ are well defined for every $\beta\in\mathbb{R}$ and
$\left\Vert \left(  I-A\right)  ^{\alpha/2}f\right\Vert _{L^{2}\left(
\mathbb{R}^{d}\right)  }$ is equivalent to the norm in $W^{\alpha,2}\left(
\mathbb{R}^{d}\right)  $. Recall also that (see \cite{Pa}), for every
$\beta>0$, and given $T>0$, there is a constant $C_{\beta,T}$ such that
\[
\left\Vert \left(  I-A\right)  ^{\beta}e^{tA}\right\Vert _{L^{2}\rightarrow
L^{2}}\leq\frac{C_{\beta,T}}{t^{\beta}}%
\]
for $t\in(0,T]$.

\subsection{Equation for the empirical measure and its mild formulation}

Starting from this section, we drop the suffix $N$ in $X_{t}^{a,N}$, $I^{a,N}%
$, $T_{i}^{a,N}$, $\lambda^{a,N}$, $\mathcal{N}_{t}^{a,N}$ to simplify
notations. Let $\delta$ denote a point outside $\mathbb{R}^{d}$, the so called
grave state, where we assume the processes $X_{t}^{a}$ live when $t\notin
I^{a}$. Hence, whenever a particle proliferates and therefore dies, it stays
forever in the grave state $\delta$. In the sequel, the test functions $\phi$
are assumed to be defined over $\mathbb{R}^{d}\cup\left\{  \delta\right\}  $
and be such that $\phi\left(  \delta\right)  =0$. Using It\^{o} formula over
random time intervals, one can show that $\phi\left(  X_{t}^{a}\right)  $,
with $\phi\in C^{2}\left(  \mathbb{R}^{d}\right)  $, satisfies%
\[
\phi\left(  X_{t}^{a}\right)  =\phi\left(  X_{T_{0}^{a}}^{a}\right)  1_{t\geq
T_{0}^{a}}-\phi\left(  X_{T_{1}^{a}}^{a}\right)  1_{t\geq T_{1}^{a}}+\sqrt
{2}\int_{0}^{t}1_{s\in I^{a}}\nabla\phi\left(  X_{s}^{a}\right)  dB_{s}%
^{a}+\int_{0}^{t}1_{s\in I^{a}}\Delta\phi\left(  X_{s}^{a}\right)  ds.
\]
With a few computations, one can see that the empirical measure $S_{t}^{N}$
satisfies%
\begin{equation}
d\left\langle S_{t}^{N},\phi\right\rangle =\left\langle S_{t}^{N},\Delta
\phi\right\rangle dt+\left\langle \left(  1-h_{t}^{N}\right)  ^{+}S_{t}%
^{N},\phi\right\rangle dt+dM_{t}^{1,\phi,N}+dM_{t}^{2,\phi,N}
\label{SPDE for empirical measure}%
\end{equation}
for every $\phi\in C_{b}^{2}\left(  \mathbb{R}^{d}\right)  $ and where%
\begin{align*}
M_{t}^{1,\phi,N}  &  :=\frac{\sqrt{2}}{N}\sum_{a\in\Lambda^{N}}\int_{0}%
^{t}1_{s\in I^{a}}\nabla\phi\left(  X_{s}^{a}\right)  \cdot\mathrm{d}B_{s}%
^{a},\\
M_{t}^{2,\phi,N}  &  :=\frac{1}{N}\sum_{a\in\Lambda^{N}}\phi\left(
X_{T_{1}^{a}}^{a}\right)  1_{t\geq T_{1}^{a}}-\frac{1}{N}\sum_{a\in\Lambda
^{N}}\int_{0}^{t}\phi\left(  X_{s}^{a}\right)  \lambda_{s}^{a}\mathrm{d}s.
\end{align*}
We deduce that $h_{t}^{N}\left(  x\right)  $ satisfies%
\[
dh_{t}^{N}\left(  x\right)  =\Delta h_{t}^{N}\left(  x\right)  dt+\left(
\theta_{N}\ast\left(  \left(  1-h_{t}^{N}\right)  ^{+}S_{t}^{N}\right)
\right)  \left(  x\right)  dt+dM_{t}^{1,N}\left(  x\right)  +dM_{t}%
^{2,N}\left(  x\right)  ,
\]
where%
\begin{align*}
M_{t}^{1,N}(x)  &  :=-\frac{\sqrt{2}}{N}\sum_{a\in\Lambda^{N}}\int_{0}^{t}%
{1}_{s\in I^{a}}\nabla\theta_{N}(x-X_{s}^{a})\cdot\mathrm{d}B_{s}^{a},\\
M_{t}^{2,N}(x)  &  :=\frac{1}{N}\sum_{a\in\Lambda^{N}}\theta_{N}%
(x-X_{T_{1}^{a}}^{a}){1}_{t\geq T_{1}^{a}}-\frac{1}{N}\int_{0}^{t}\sum
_{a\in\Lambda_{s}^{N}}\theta_{N}(x-X_{s}^{a})\lambda_{s}^{a}\mathrm{d}s\\
&  =\frac{1}{N}\sum_{a\in\Lambda^{N}}\int_{0}^{t}\theta_{N}\left(
x-X_{s-}^{a}\right)  \mathrm{d}\left(  \mathcal{N}_{s}^{a}-\Lambda_{s}%
^{a}\right)  .
\end{align*}
Following a standard procedure, used for instance by \cite{DaPrZab}, we may
rewrite this equation in mild form:%
\begin{equation}
h_{t}^{N}=e^{tA}h_{0}^{N}+\int_{0}^{t}e^{\left(  t-s\right)  A}\left(
\theta_{N}\ast\left(  \left(  1-h_{s}^{N}\right)  ^{+}S_{s}^{N}\right)
\right)  ds+\int_{0}^{t}e^{\left(  t-s\right)  A}dM_{s}^{1,N}+\int_{0}%
^{t}e^{\left(  t-s\right)  A}dM_{s}^{2,N}. \label{mild form}%
\end{equation}
This opens the possibility of a semigroup approach, which is a main novelty of
this paper.

\subsection{Total mass and useful inequalities}

The total relative mass
\[
\left[  S_{t}^{N}\right]  :=S_{t}^{N}\left(  \mathbb{R}^{d}\right)
=\left\langle S_{t}^{N},1\right\rangle =\frac{Card\left(  \Lambda_{t}%
^{N}\right)  }{N}%
\]
plays a central role. Since, in our model, the number of particles may only
increase, we have
\begin{equation}
\left[  S_{t}^{N}\right]  \leq\left[  S_{T}^{N}\right]  \qquad\text{for all
}t\in\left[  0,T\right]  .\label{monotone}%
\end{equation}
The quantity $\left[  S_{T}^{N}\right]  $ is, moreover, exponentially
integrable, uniformly in $N$, see Lemma \ref{lemma total mass} below. We also
repeatedly use the identity%
\begin{equation}
\int_{\mathbb{R}^{d}}h_{t}^{N}(x)\mathrm{d}x=\left[  S_{t}^{N}\right]
,\label{h equal S}%
\end{equation}
which follows from Fubini theorem. Another simple rule of calculus we often
use is%
\begin{equation}
\left\vert \left(  \theta_{N}\ast\left(  fS_{t}^{N}\right)  \right)  \left(
x\right)  \right\vert \leq\left\Vert f\right\Vert _{\infty}h_{t}^{N}\left(
x\right)  \label{rules on f}%
\end{equation}
for every bounded measurable $f:\mathbb{R}^{d}\rightarrow\mathbb{R}$.
Moreover, since $h_{s}^{N}\geq0$, we have
\begin{equation}
\left(  1-h_{s}^{N}\left(  x\right)  \right)  ^{+}\in\left[  0,1\right]
.\label{plus}%
\end{equation}
Finally, we often use the inequality%
\begin{equation}
\frac{1}{N}\int_{\mathbb{R}^{d}}\left\vert \theta_{N}(x)\right\vert ^{2}dx\leq
C,\label{bound on theta N}%
\end{equation}
which holds with a suitable constant $C>0$. Indeed, it holds
\[
\frac{1}{N}\int_{\mathbb{R}^{d}}\left\vert \theta_{N}(x)\right\vert
^{2}dx=\frac{\epsilon_{N}^{-d}}{N}\int_{\mathbb{R}^{d}}\epsilon_{N}%
^{-d}\left\vert \theta\left(  \epsilon_{N}^{-1}x\right)  \right\vert
^{2}dx=\frac{\epsilon_{N}^{-d}}{N}\int_{\mathbb{R}^{d}}\left\vert
\theta\left(  x\right)  \right\vert ^{2}dx.
\]
Inequality (\ref{bound on theta N}) follows from the assumptions $\theta\in
L^{2}(\mathbb{R}^{d})$ and $\sup_{N}\epsilon_{N}^{-d}/N<\infty$.

\section{Main estimates on martingale terms}

Let $\alpha\in(d/2,\alpha_{0})$, as in the statement of Theorem \ref{Thm 1}.

\begin{lemma}
\label{lemma martingale 1} There exists a constant $C^{\prime}>0$ such that
for all $N\in\mathbb{N}$, $t\in\lbrack0,T]$, small $h>0$
\[
\left\Vert \int_{0}^{t}\left(  I-A\right)  ^{\frac{\alpha}{2}}e^{\left(
t+h-s\right)  A}\mathrm{d}M_{s}^{1,N}\right\Vert _{L^{2}\left(  \Omega
\times\mathbb{R}^{d}\right)  }\leq C^{\prime}.
\]

\end{lemma}

\begin{proof}%
\begin{align*}
&  \left\Vert \int_{0}^{t}\left(  I-A\right)  ^{\alpha/2}e^{\left(
t+h-s\right)  A}dM_{s}^{1,N}\right\Vert _{L^{2}\left(  \Omega\times
\mathbb{R}^{d}\right)  }^{2}\\
&  =\frac{2}{N^{2}}\int_{\mathbb{R}^{d}}E\left[  \left\vert \sum_{a\in
\Lambda^{N}}\int_{0}^{t}\left(  \left(  I-A\right)  ^{\alpha/2}e^{\left(
t+h-s\right)  A}{1}_{s\in I^{a}}\nabla\theta_{N}(\cdot-X_{s}^{a})\right)
\left(  x\right)  \cdot\mathrm{d}B_{s}^{a}\right\vert ^{2}\right]  dx\\
&  =\frac{2}{N^{2}}\int_{\mathbb{R}^{d}}E\left[  \sum_{a\in\Lambda^{N}}%
\int_{0}^{t}\left\vert \left(  \left(  I-A\right)  ^{\alpha/2}e^{\left(
t+h-s\right)  A}{1}_{s\in I^{a}}\nabla\theta_{N}(\cdot-X_{s}^{a})\right)
\left(  x\right)  \right\vert ^{2}ds\right]  dx\\
&  =\frac{2}{N^{2}}E\left[  \sum_{a\in\Lambda^{N}}\int_{0}^{t}{1}_{s\in I^{a}%
}\left(  \int_{\mathbb{R}^{d}}\left\vert \left(  \left(  I-A\right)
^{\alpha/2}e^{\left(  t+h-s\right)  A}\nabla\theta_{N}(\cdot-X_{s}%
^{a})\right)  \left(  x\right)  \right\vert ^{2}dx\right)  ds\right]  .
\end{align*}
We have%
\[
\left(  \left(  I-A\right)  ^{\alpha/2}e^{\left(  t+h-s\right)  A}\nabla
\theta_{N}(\cdot-X_{s}^{a})\right)  \left(  x\right)  =\left(  \left(
I-A\right)  ^{\alpha/2}e^{\left(  t+h-s\right)  A}\nabla\theta_{N})\right)
\left(  x-X_{s}^{a}\right)  .
\]
Then, by change of variable,
\begin{align*}
&  \int_{\mathbb{R}^{d}}\left\vert \left(  \left(  I-A\right)  ^{\alpha
/2}e^{\left(  t+h-s\right)  A}\nabla\theta_{N}(\cdot-X_{s}^{a})\right)
\left(  x\right)  \right\vert ^{2}dx\\
&  =\int_{\mathbb{R}^{d}}\left\vert \left(  \left(  I-A\right)  ^{\alpha
/2}e^{\left(  t+h-s\right)  A}\nabla\theta_{N}\right)  \left(  x\right)
\right\vert ^{2}dx.
\end{align*}
Therefore, since $\frac{1}{N}\sum_{a\in\Lambda^{N}}{1}_{s\in I^{a}}=\left[
S_{s}^{N}\right]  \leq\left[  S_{T}^{N}\right]  $,
\begin{align*}
&  \left\Vert \int_{0}^{t}\left(  I-A\right)  ^{\alpha/2}e^{\left(
t+h-s\right)  A}dM_{s}^{1,N}\right\Vert _{L^{2}\left(  \Omega\times
\mathbb{R}^{d}\right)  }^{2}\\
&  =\frac{2}{N}E\left[  \int_{0}^{t}\left(  \frac{1}{N}\sum_{a\in\Lambda^{N}%
}{1}_{s\in I^{a}}\right)  \left(  \int_{\mathbb{R}^{d}}\left\vert \left(
\left(  I-A\right)  ^{\alpha/2}e^{\left(  t+h-s\right)  A}\nabla\theta
_{N}\right)  \left(  x\right)  \right\vert ^{2}dx\right)  ds\right] \\
&  \leq\frac{2}{N}E\left(  \left[  S_{T}^{N}\right]  \right)  \int_{0}%
^{t}\left\Vert \left(  I-A\right)  ^{\alpha/2}e^{\left(  t+h-s\right)
A}\nabla\theta_{N}\right\Vert _{L^{2}}^{2}ds.
\end{align*}
From assumption (\ref{assumption on theta}) and the condition $\alpha
\in(d/2,\alpha_{0})$, we have $\frac{\beta}{d}\left(  2\alpha+d\right)  <1$,
hence there exists a small $\varepsilon>0$ such that $\frac{\beta}{d}\left(
2\alpha+\varepsilon+d\right)  \leq1$ and at the same time $\alpha
+\frac{\varepsilon}{2}\leq\alpha_{0}$. Denoting by $C>0$ any constant
independent of $N$ and recalling that $\epsilon_{N}=N^{-\frac{\beta}{d}}$, we
have%
\begin{align*}
&  \leq\frac{C}{N}\int_{0}^{t}\left\Vert \left(  I-A\right)  ^{\left(
1-\varepsilon/2\right)  /2}e^{\left(  t-s\right)  A}\right\Vert _{L^{2}%
\rightarrow L^{2}}^{2}\left\Vert \nabla\left(  I-A\right)  ^{-1/2}\right\Vert
_{L^{2}\rightarrow L^{2}}^{2}\left\Vert \left(  I-A\right)  ^{\left(
\alpha+\varepsilon/2\right)  /2}e^{hA}\theta_{N}\right\Vert _{L^{2}}^{2}ds\\
&  \leq\frac{C}{N}\left\Vert \theta_{N}\right\Vert _{W^{\alpha+\varepsilon
/2,2}}^{2}\int_{0}^{t}\frac{1}{\left(  t-s\right)  ^{1-\varepsilon/2}}ds\leq
C\frac{\epsilon_{N}^{-2\alpha-\varepsilon-d}}{N}\leq C
\end{align*}
where we have used Lemma \ref{lemma interpolation} below.
\end{proof}


\begin{lemma}
\label{lemma martingale 2} There exists a constant $C>0$ such that for all
$N\in\mathbb{N}$, $t\in\lbrack0,T]$, small $h>0$
\[
\left\Vert \int_{0}^{t}\left(  I-A\right)  ^{\frac{\alpha}{2}}e^{\left(
t+h-s\right)  A}\mathrm{d}M_{s}^{2,N}\right\Vert _{L^{2}\left(  \Omega
\times\mathbb{R}^{d}\right)  }\leq C.
\]

\end{lemma}

\begin{proof}
Since%
\[
M_{t}^{2,N}=\frac{1}{N}\sum_{a\in\Lambda^{N}}\int_{0}^{t}\theta_{N}\left(
x-X_{s-}^{a}\right)  \mathrm{d}\left(  \mathcal{N}_{s}^{a}-\Lambda_{s}%
^{a}\right)
\]
we have%
\begin{align*}
&  \left\Vert \int_{0}^{t}\left(  I-A\right)  ^{\alpha/2}e^{\left(
t+h-s\right)  A}dM_{s}^{2,N}\right\Vert _{L^{2}\left(  \Omega\times
\mathbb{R}^{d}\right)  }^{2}\\
&  =\frac{1}{N^{2}}\int_{\mathbb{R}^{d}}E\left[  \left\vert \sum_{a\in
\Lambda^{N}}\int_{0}^{t}\left(  \left(  I-A\right)  ^{\alpha/2}e^{\left(
t+h-s\right)  A}{1}_{s\in I^{a}}\nabla\theta_{N}(\cdot-X_{s-}^{a})\right)
\left(  x\right)  \cdot\mathrm{d}\left(  \mathcal{N}_{s}^{a}-\Lambda_{s}%
^{a}\right)  \right\vert ^{2}\right]  dx.
\end{align*}
Write $g_{t,s,h}^{a,N}\left(  X_{s-}^{a}\right)  $ for $\left(  \left(
I-A\right)  ^{\alpha/2}e^{\left(  t+h-s\right)  A}{1}_{s\in I^{a}}\nabla
\theta_{N}(\cdot-X_{s-}^{a})\right)  \left(  x\right)  $. Since the jumps of
$\mathcal{N}_{s}^{a}$ and $\mathcal{N}_{s}^{a^{\prime}}$, for $a\neq
a^{\prime}$, never occur at the same time, we have%
\[
E\left[  \left(  \int_{0}^{t}g_{t,s,h}^{a,N}\left(  X_{s-}^{a}\right)
\mathrm{d}\left(  \mathcal{N}_{s}^{a}-\Lambda_{s}^{a}\right)  \right)  \left(
\int_{0}^{t}g_{t,s,h}^{a,N}\left(  X_{s-}^{a^{\prime}}\right)  \mathrm{d}%
\left(  \mathcal{N}_{s}^{a^{\prime}}-\Lambda_{s}^{a^{\prime},N}\right)
\right)  \right]  =0.
\]
Hence the last expression is equal to%
\[
=\frac{1}{N^{2}}\sum_{a\in\Lambda^{N}}\int_{\mathbb{R}^{d}}E\left[  \left\vert
\int_{0}^{t}g_{t,s,h}^{a,N}\left(  X_{s-}^{a}\right)  \cdot\mathrm{d}\left(
\mathcal{N}_{s}^{a}-\Lambda_{s}^{a}\right)  \right\vert ^{2}\right]  dx.
\]
It is known that%
\[
E\left[  \left\vert \int_{0}^{t}g_{t,s,h}^{a,N}\left(  X_{s-}^{a}\right)
\mathrm{d}\left(  \mathcal{N}_{s}^{a}-\Lambda_{s}^{a}\right)  \right\vert
^{2}\right]  =E\left[  \int_{0}^{t}\left\vert g_{t,s,h}^{a,N}\left(  X_{s}%
^{a}\right)  \right\vert ^{2}\mathrm{d}\Lambda_{s}^{a}\right]  .
\]
Therefore, the last expression simplifies to%
\begin{align*}
&  =\frac{1}{N^{2}}\sum_{a\in\Lambda^{N}}\int_{\mathbb{R}^{d}}E\left[
\int_{0}^{t}\left\vert g_{t,s,h}^{a,N}\left(  X_{s}^{a}\right)  \right\vert
^{2}\lambda_{s}^{a}\mathrm{d}s\right]  dx\\
&  =\frac{1}{N^{2}}\sum_{a\in\Lambda^{N}}E\left[  \int_{0}^{t}\left(
\int_{\mathbb{R}^{d}}\left\vert \left(  \left(  I-A\right)  ^{\alpha
/2}e^{\left(  t+h-s\right)  A}{1}_{s\in I^{a}}\nabla\theta_{N}(\cdot-X_{s}%
^{a})\right)  \left(  x\right)  \right\vert ^{2}dx\right)  \lambda_{s}%
^{a}\mathrm{d}s\right]  .
\end{align*}
As in the previous proof, and taking into account the boundedness of
$\lambda_{s}^{a}$ (by definition),
\begin{align*}
&  =\frac{1}{N^{2}}\sum_{a\in\Lambda^{N}}E\left[  \int_{0}^{t}\left(
\int_{\mathbb{R}^{d}}\left\vert \left(  \left(  I-A\right)  ^{\alpha
/2}e^{\left(  t+h-s\right)  A}{1}_{s\in I^{a}}\nabla\theta_{N}\right)  \left(
x\right)  \right\vert ^{2}dx\right)  \lambda_{s}^{a}\mathrm{d}s\right] \\
&  \leq\frac{1}{N}E\left[  \int_{0}^{t}\left(  \frac{1}{N}\sum_{a\in
\Lambda^{N}}{1}_{s\in I^{a}}\right)  \left\Vert \left(  I-A\right)
^{\alpha/2}e^{\left(  t+h-s\right)  A}\nabla\theta_{N}\right\Vert _{L^{2}}%
^{2}\mathrm{d}s\right] \\
&  \leq\frac{1}{N}E\left(  \left[  S_{T}^{N}\right]  \right)  \int_{0}%
^{t}\left\Vert \left(  I-A\right)  ^{\alpha/2}e^{\left(  t+h-s\right)
A}\nabla\theta_{N}\right\Vert _{L^{2}}^{2}ds.
\end{align*}
This is the same expression as in the previous proof, which is bounded by a
constant, uniformly in $N$.
\end{proof}

\section{Main estimate on $h_{t}^{N}$\label{section main estimate}}

As described above, we need an estimate on $h_{t}^{N}$ in a H\"{o}lder norm
(in space) which we gain by Sobolev embedding theorem. Since we work in an
$L^{2}$-setting (computations not reported here in the $L^{p}$ setting do not
help since they re-introduce difficulties from other sides), we have%
\[
W^{\alpha,2}\left(  \mathbb{R}^{d}\right)  \subset C_{b}^{\varepsilon}\left(
\mathbb{R}^{d}\right)  \qquad\text{if }\left(  \alpha-\varepsilon\right)
2\geq d.
\]
This is the reason for the restriction on $\alpha$, namely $2\alpha>d$. Recall
that $\alpha_{0}$ and $\rho_{0}$ were introduced in (\ref{assumption on theta}%
) and \emph{(\ref{initial cond}) respectively.}

\begin{lemma}
\label{lemma p equal 2} Assume $\alpha\in(d/2,\alpha_{0})$. Then there exist
constants $C,C^{\prime}>0$ such that for all $N\in\mathbb{N},t\in(0,T]$
\[
\left\Vert h_{t}^{N}\right\Vert _{L^{2}\left(  \Omega;W^{\alpha,2}\left(
\mathbb{R}^{d}\right)  \right)  }\leq C\mathbb{E}\left[  \left\Vert \left(
I-A\right)  ^{\frac{\alpha}{2}}h_{t}^{N}\right\Vert _{L^{2}\left(
\mathbb{R}^{d}\right)  }^{2}\right]  ^{1/2}\leq C^{\prime}\left(  1+\frac
{1}{t^{\frac{\left(  \alpha-\rho_{0}\right)  \vee0}{2}}}\right)  .
\]

\end{lemma}

\begin{proof}
The first inequality follows from the fact that the two norms
\[
\left\Vert \cdot\right\Vert _{W^{\alpha,2}(\mathbb{R}^{d})}\text{ and
}\left\Vert \left(  I-A\right)  ^{\frac{\alpha}{2}}\cdot\right\Vert
_{L^{2}(\mathbb{R}^{d})}%
\]
are equivalent. From the mild formulation (\ref{mild form}) we have%
\begin{align*}
&  \left\Vert \left(  I-A\right)  ^{\alpha/2}e^{hA}h_{t}^{N}\right\Vert
_{L^{2}\left(  \Omega\times\mathbb{R}^{d}\right)  }\\
&  \leq\left\Vert \left(  I-A\right)  ^{\alpha/2}e^{(t+h)A}h_{0}%
^{N}\right\Vert _{L^{2}\left(  \Omega\times\mathbb{R}^{d}\right)  }\\
&  +\int_{0}^{t}\left\Vert \left(  I-A\right)  ^{\alpha/2}e^{\left(
t+h-s\right)  A}\left(  \theta_{N}\ast\left(  \left(  1-h_{s}^{N}\right)
^{+}S_{s}^{N}\right)  \right)  \right\Vert _{L^{2}\left(  \Omega
\times\mathbb{R}^{d}\right)  }ds\\
&  +\left\Vert \int_{0}^{t}\left(  I-A\right)  ^{\alpha/2}e^{\left(
t+h-s\right)  A}dM_{s}^{1,N}\right\Vert _{L^{2}\left(  \Omega\times
\mathbb{R}^{d}\right)  }+\left\Vert \int_{0}^{t}\left(  I-A\right)
^{\alpha/2}e^{\left(  t+h-s\right)  A}dM_{s}^{2,N}\right\Vert _{L^{2}\left(
\Omega\times\mathbb{R}^{d}\right)  }.
\end{align*}
The last two terms are bounded by a constant, by Lemmata
\ref{lemma martingale 1} and \ref{lemma martingale 2}. For the first term,
where $C>0$ is a constant that may change from instance to instance, we have
\begin{align*}
&  \left\Vert \left(  I-A\right)  ^{\alpha/2}e^{(t+h)A}h_{0}^{N}\right\Vert
_{L^{2}\left(  \Omega\times\mathbb{R}^{d}\right)  }\\
&  \leq\left\Vert \left(  I-A\right)  ^{\left(  \alpha-\rho_{0}\right)
/2}e^{(t+h)A}\right\Vert _{L^{2}\left(  \mathbb{R}^{d}\right)  \rightarrow
L^{2}\left(  \mathbb{R}^{d}\right)  }\left\Vert \left(  I-A\right)  ^{\rho
_{0}/2}h_{0}^{N}\right\Vert _{L^{2}\left(  \Omega\times\mathbb{R}^{d}\right)
}\\
&  \leq\frac{C}{(t+h)^{\frac{\left(  \alpha-\rho_{0}\right)  \vee0}{2}}}%
\end{align*}
where we have used assumption (\ref{initial cond}). About the second one,
\begin{align*}
&  \int_{0}^{t}\left\Vert \left(  I-A\right)  ^{\alpha/2}e^{\left(
t+h-s\right)  A}\left(  \theta_{N}\ast\left(  \left(  1-h_{s}^{N}\right)
^{+}S_{s}^{N}\right)  \right)  \right\Vert _{L^{2}\left(  \Omega
\times\mathbb{R}^{d}\right)  }ds\\
&  \leq\int_{0}^{t}\left\Vert e^{\left(  t-s\right)  A}\right\Vert
_{L^{2}\left(  \mathbb{R}^{d}\right)  \rightarrow L^{2}\left(  \mathbb{R}%
^{d}\right)  }\left\Vert \left(  I-A\right)  ^{\alpha/2}e^{hA}\left(
\theta_{N}\ast\left(  \left(  1-h_{s}^{N}\right)  ^{+}S_{s}^{N}\right)
\right)  \right\Vert _{L^{2}\left(  \Omega\times\mathbb{R}^{d}\right)  }ds.
\end{align*}
Since the operator $f\mapsto\left(  I-A\right)  ^{\alpha/2}e^{hA}f$ is
positive on $L^{2}\left(  \mathbb{R}^{d}\right)  $, see Lemma
\ref{lemma positive}, it holds $\left(  I-A\right)  ^{\alpha/2}e^{hA}%
f\leq\left(  I-A\right)  ^{\alpha/2}e^{hA}g$ if $f\leq g$. Because of
\[
0\leq\left(  \theta_{N}\ast\left(  \left(  1-h_{s}^{N}\right)  ^{+}S_{s}%
^{N}\right)  \right)  \left(  x\right)  \leq h_{s}^{N}\left(  x\right)  ,
\]
we deduce
\[
0\leq\left(  I-A\right)  ^{\alpha/2}e^{hA}\left(  \theta_{N}\ast\left(
\left(  1-h_{s}^{N}\right)  ^{+}S_{s}^{N}\right)  \right)  \leq\left(
I-A\right)  ^{\alpha/2}e^{hA}h_{s}^{N}.
\]
Hence,
\begin{align*}
&  \int_{0}^{t}\left\Vert \left(  I-A\right)  ^{\alpha/2}e^{\left(
t+h-s\right)  A}\left(  \theta_{N}\ast\left(  \left(  1-h_{s}^{N}\right)
^{+}S_{s}^{N}\right)  \right)  \right\Vert _{L^{2}\left(  \Omega
\times\mathbb{R}^{d}\right)  }ds\\
&  \leq C\int_{0}^{t}\left\Vert \left(  I-A\right)  ^{\alpha/2}e^{hA}h_{s}%
^{N}\right\Vert _{L^{2}\left(  \Omega\times\mathbb{R}^{d}\right)  }ds.
\end{align*}
Until now we have proved%
\[
\left\Vert \left(  I-A\right)  ^{\alpha/2}e^{hA}h_{t}^{N}\right\Vert
_{L^{2}\left(  \Omega\times\mathbb{R}^{d}\right)  }\leq\frac{C}{(t+h)^{\frac
{\left(  \alpha-\rho_{0}\right)  \vee0}{2}}}+C\int_{0}^{t}\left\Vert \left(
I-A\right)  ^{\alpha/2}e^{hA}h_{s}^{N}\right\Vert _{L^{2}\left(  \Omega
\times\mathbb{R}^{d}\right)  }ds+C.
\]
By Gronwall's lemma we deduce
\[
\left\Vert \left(  I-A\right)  ^{\alpha/2}e^{hA}h_{t}^{N}\right\Vert
_{L^{2}\left(  \Omega\times\mathbb{R}^{d}\right)  }\leq\frac{C}{(t+h)^{\frac
{\left(  \alpha-\rho_{0}\right)  \vee0}{2}}}+C.
\]
We may now take the limit as $h\rightarrow0$. The proof is complete.
\end{proof}

\begin{remark}
The result is true also for $\alpha=0$:\
\begin{equation}
\sup_{t\in\left[  0,T\right]  }E\left[  \left\Vert h_{t}^{N}\right\Vert
_{L^{2}\left(  \mathbb{R}^{d}\right)  }^{2}\right]  \leq C. \label{L2 est}%
\end{equation}

\end{remark}

\section{Other estimates on $h_{t}^{N}$\label{section other estimates}}

In order to show tightness of the family of the functions $\{h^{N}\}_{N}$, in
addition to the previous bound which shows a regularity in space, we also need
a regularity in time. See the compactness criteria below.

\begin{lemma}
\label{lemma fractional Sobolev}Given any $\gamma\in(0,1/2)$, it holds
\[
\lim_{R\rightarrow\infty}\sup_{N\in\mathbb{N}}P\left(  \int_{0}^{T}\int%
_{0}^{T}\frac{\left\Vert h_{t}^{N}-h_{s}^{N}\right\Vert _{W^{-2,2}}^{2}%
}{|t-s|^{1+2\gamma}}\mathrm{d}s\mathrm{d}t>R\right)  =0.
\]

\end{lemma}

\begin{proof}
\textbf{Step 1. }We need to estimate $\left\Vert h_{t}^{N}-h_{s}%
^{N}\right\Vert _{W^{-2,2}}^{2}$ in such a way that it cancels with the
singularity in the denominator at $t=s$. Notice that $L^{2}\subset W^{-2,2}$
with continuous embedding, namely there exists a constant $C>0$ such that
$\left\Vert f\right\Vert _{W^{-2,2}}\leq C\left\Vert f\right\Vert _{L^{2}}$
for all $f\in L^{2}$;\ similarly for $W^{-1,2}\subset W^{-2,2}$. Moreover, the
linear operator $\Delta$ is bounded from $L^{2}$ to $W^{-2,2}$. Therefore (we
denote by $C>0$ any constant independent of $N$, $h_{.}^{N}$, $t$, $s$)
\begin{align*}
\left\Vert h_{t}^{N}-h_{s}^{N}\right\Vert _{W^{-2,2}}^{2}  &  \leq C\left\Vert
\int_{s}^{t}\Delta h_{r}^{N}\mathrm{d}r\right\Vert _{W^{-2,2}}^{2}+C\left\Vert
\int_{s}^{t}\theta_{N}\ast\left(  \left(  1-\theta_{N}\ast S_{r}^{N}\right)
S_{r}^{N}\right)  \mathrm{d}r\right\Vert _{W^{-2,2}}^{2}\\
&  +C\left\Vert M_{t}^{1,N}-M_{s}^{1,N}\right\Vert _{W^{-2,2}}^{2}+C\left\Vert
M_{t}^{2,N}-M_{s}^{2,N}\right\Vert _{W^{-2,2}}^{2}%
\end{align*}
hence by H\"{o}lder inequality%
\begin{align*}
&  \leq C\left(  t-s\right)  \int_{s}^{t}\left\Vert h_{r}^{N}\right\Vert
_{L^{2}}^{2}\mathrm{d}r+C\left(  t-s\right)  \int_{s}^{t}\left\Vert \theta
_{N}\ast\left(  \left(  1-\theta_{N}\ast S_{r}^{N}\right)  S_{r}^{N}\right)
\right\Vert _{L^{2}}^{2}\mathrm{d}r\\
&  +C\left\Vert M_{t}^{1,N}-M_{s}^{1,N}\right\Vert _{W^{-1,2}}^{2}+C\left\Vert
M_{t}^{2,N}-M_{s}^{2,N}\right\Vert _{L^{2}}^{2}%
\end{align*}
and now using (\ref{rules on f})%
\[
\leq C\left(  t-s\right)  \int_{s}^{t}\left\Vert h_{r}^{N}\right\Vert _{L^{2}%
}^{2}\mathrm{d}r+C\left\Vert M_{t}^{1,N}-M_{s}^{1,N}\right\Vert _{W^{-1,2}%
}^{2}{}^{2}+C\left\Vert M_{t}^{2,N}-M_{s}^{2,N}\right\Vert _{L^{2}}^{2}.
\]
Accordingly, we split the estimate of $P\left(  \int_{0}^{T}\int_{0}^{T}%
\frac{\left\Vert h_{t}^{N}-h_{s}^{N}\right\Vert _{W^{-2,2}}^{2}}%
{|t-s|^{1+2\gamma}}\mathrm{d}s\mathrm{d}t>R\right)  $ in three more elementary
estimates, that now we handle separately; the final result will be a
consequence of them.

The number $C_{\gamma}=\int_{0}^{T}\int_{0}^{T}\frac{1}{|t-s|^{2\gamma}%
}\mathrm{d}s\mathrm{d}t$ is finite, hence the first addend is bounded by
(renaming the constant $C$)%
\begin{align*}
&  P\left(  \int_{0}^{T}\int_{0}^{T}\frac{C\left(  t-s\right)  \left(  \left[
S_{T}^{N}\right]  +1\right)  \sup_{r\in\left[  0,T\right]  }\left\Vert
h_{r}^{N}\right\Vert _{L^{2}}^{2}}{|t-s|^{1+2\gamma}}\mathrm{d}s\mathrm{d}%
t>R\right) \\
&  =P\left(  \left(  \left[  S_{T}^{N}\right]  +1\right)  \sup_{r\in\left[
0,T\right]  }\left\Vert h_{r}^{N}\right\Vert _{L^{2}}^{2}>R/C\right) \\
&  \leq P\left(  \left[  S_{T}^{N}\right]  +1>\sqrt{R/C}\right)  +P\left(
\sup_{r\in\left[  0,T\right]  }\left\Vert h_{r}^{N}\right\Vert _{L^{2}}%
^{2}>\sqrt{R/C}\right)
\end{align*}
and both these terms are, uniformly in $N$, small for large $R$, due to Lemma
\ref{lemma total mass} and estimate (\ref{L2 est}).

\textbf{Step 2}. Concerning the martingale terms, we now prove that%
\[
E\left\Vert M_{t}^{1,N}-M_{s}^{1,N}\right\Vert _{W^{-1,2}}^{2}\leq C\left\vert
t-s\right\vert
\]
and%
\[
E\left\Vert M_{t}^{2,N}-M_{s}^{2,N}\right\Vert _{L^{2}}\leq C\left\vert
t-s\right\vert
\]
for some constant $C>0$. By Chebyshev's inequality it follows that%
\[
\lim_{R\rightarrow\infty}\sup_{N\in\mathbb{N}}P\left(  \int_{0}^{T}\int%
_{0}^{T}\frac{\left\Vert M_{t}^{1,N}-M_{s}^{1,N}\right\Vert _{W^{-1,2}}^{2}%
}{|t-s|^{1+2\gamma}}\mathrm{d}s\mathrm{d}t>R\right)  =0,
\]%
\[
\lim_{R\rightarrow\infty}\sup_{N\in\mathbb{N}}P\left(  \int_{0}^{T}\int%
_{0}^{T}\frac{\left\Vert M_{t}^{2,N}-M_{s}^{2,N}\right\Vert _{L^{2}}^{2}%
}{|t-s|^{1+2\gamma}}\mathrm{d}s\mathrm{d}t>R\right)  =0
\]
and the proof will be complete. For notational convenience, we abbreviate, for
$=1,2$,%
\[
M_{t}^{i,N}(x)=\frac{1}{N}\sum_{a\in A^{N}}M_{t}^{i,a}(x).
\]
Note, that for every $x\in\mathbb{R}^{d}$ the processes $M^{1,N}(x)$ and
$M^{2,N}(x)$ are martingales. It follows, with computations similar to those
of Lemma \ref{lemma martingale 1}, for $t\geq s$
\begin{align*}
E\left\Vert M_{t}^{1,N}-M_{s}^{1,N}\right\Vert _{W^{-1,2}}^{2}  &
=\int_{\mathbb{R}^{d}}\frac{1}{N^{2}}\sum_{a\in A^{N}}E\left[  \int_{s}%
^{t}1_{r\in I^{a}}(I-\Delta)^{-\frac{1}{2}}\nabla\theta_{N}\left(  x-X_{r}%
^{a}\right)  ^{2}\mathrm{d}r\right]  \mathrm{d}x\\
&  =\frac{1}{N}\left\Vert (I-\Delta)^{-\frac{1}{2}}\nabla\theta_{N}\right\Vert
_{L^{2}}^{2}E\int_{s}^{t}\frac{1}{N}\sum_{a\in A^{N}}1_{r\in I^{a}}%
\mathrm{d}r\\
&  \leq\frac{1}{N}\left\Vert \theta_{N}\right\Vert _{L^{2}}^{2}\left(
t-s\right)  \leq C\left(  t-s\right)  .
\end{align*}
Similarly, for the second martingale, in analogy with Lemma
\ref{lemma martingale 2},
\begin{align*}
E\left\Vert M_{t}^{2,N}-M_{s}^{2,N}\right\Vert _{L^{2}}^{2}  &  =\int%
_{\mathbb{R}^{d}}\frac{1}{N^{2}}\sum_{a\in A^{N}}E\left[  M_{t}^{2,a}%
(x)^{2}-M_{s}^{2,a}(x)^{2}\right]  \mathrm{d}x\\
&  =\int_{\mathbb{R}^{d}}\frac{1}{N^{2}}\sum_{a\in A^{N}}E\left[  \int_{s}%
^{t}1_{r\in I^{a}}\theta_{N}\left(  x-X_{r}^{a}\right)  ^{2}\lambda_{r}%
^{a}\mathrm{d}r\right]  \mathrm{d}x\\
&  \leq C_{F}\frac{1}{N}\left\Vert \theta_{N}\right\Vert _{L^{2}}^{2}E\int%
_{s}^{t}\frac{1}{N}\sum_{a\in A^{N}}1_{r\in I^{a}}\mathrm{d}r\leq C\left(
t-s\right)  .
\end{align*}

\end{proof}

\section{Passage to the limit\label{sect convergence}}

\subsection{Criterion of compactness\label{subsect compactness}}

A version of Aubin-Lions lemma, see \cite{Lions}, \cite{FlaGat},
\cite{BanasBrz}, states that when $E_{0}\subset E\subset E_{1}$ are three
Banach spaces with continuous dense embedding, $E_{0},E_{1}$ reflexive, with
$E_{0}$ compactly embedded into $E$, given $p,q\in\left(  1,\infty\right)  $
and $\gamma\in\left(  0,1\right)  $, the space $L^{q}\left(  0,T;E_{0}\right)
\cap W^{\gamma,p}\left(  0,T;E_{1}\right)  $ is compactly embedded into
$L^{q}\left(  0,T;E\right)  $.

Given the number $\alpha_{0}$ in assumption (\ref{assumption on theta}), we
take any pair $\alpha^{\prime}<\alpha$ in the interval $(d/2,\alpha_{0})$. We
use Aubin-Lions lemma with $E=W^{\alpha^{\prime},2}\left(  D\right)  $,
$E_{0}=W^{\alpha,2}\left(  D\right)  $, $0<\gamma<\frac{1}{2}$ and
$E_{1}=W^{-2,2}\left(  \mathbb{R}^{d}\right)  $, where $D$ is any regular
bounded domain. The lemma states that $L^{2}\left(  0,T;W^{\alpha,2}\left(
D\right)  \right)  \cap W^{\gamma,2}\left(  0,T;W^{-2,2}\left(  \mathbb{R}%
^{d}\right)  \right)  $ is compactly embedded into $L^{2}\left(
0,T;W^{\alpha^{\prime},2}\left(  D\right)  \right)  $.

Notice that for $\gamma p>1$, the space $W^{\gamma,p}\left(  0,T;E_{1}\right)
$ is embedded into $C\left(  \left[  0,T\right]  ;E_{1}\right)  $, so it is
not suitable for our purposes since we have to deal with discontinuous
processes. However, for $\gamma p<1$ the space $W^{\gamma,p}\left(
0,T;E_{1}\right)  $ includes piecewise constant functions, as one can easily
check. Therefore it is a good space for c\`{a}dl\`{a}g processes.

Now, consider the space
\[
Y_{0}:=L^{\infty}\left(  0,T;L^{2}\left(  \mathbb{R}^{d}\right)  \right)  \cap
L^{2}\left(  0,T;W^{\alpha,2}\left(  \mathbb{R}^{d}\right)  \right)  \cap
W^{\gamma,2}\left(  0,T;W^{-2,2}\left(  \mathbb{R}^{d}\right)  \right)  .
\]
Using the Fr\'{e}chet topology on $L^{2}\left(  0,T;W_{loc}^{\alpha^{\prime
},2}\left(  \mathbb{R}^{d}\right)  \right)  $ defined as
\[
d\left(  f,g\right)  =\sum_{n=1}^{\infty}2^{-n}\left(  1\wedge\int_{0}%
^{T}\left\Vert (f-g)\left(  t,\cdot\right)  \right\Vert _{W^{\alpha^{\prime
},2}\left(  B\left(  0,n\right)  \right)  }^{p}dt\right)
\]
one has that $L^{2}\left(  0,T;W^{\alpha,2}\left(  \mathbb{R}^{d}\right)
\right)  \cap W^{\gamma,2}\left(  0,T;W^{-2,2}\left(  \mathbb{R}^{d}\right)
\right)  $ is compactly embedded into $L^{2}\left(  0,T;W_{loc}^{\alpha
^{\prime},2}\left(  \mathbb{R}^{d}\right)  \right)  $ (the proof is
elementary, using the fact that if a set is compact in $L^{2}\left(
0,T;W_{loc}^{\alpha^{\prime},2}\left(  B\left(  0,n\right)  \right)  \right)
$ for every $n$ then it is compact in $L^{2}\left(  0,T;W_{loc}^{\alpha
^{\prime},2}\left(  \mathbb{R}^{d}\right)  \right)  $\ with this topology; see
a similar result in \cite{BrzMot}). Denoting by $L_{w\ast}^{\infty}\left(
0,T;L^{2}\left(  \mathbb{R}^{d}\right)  \right)  $ and $L_{w}^{2}\left(
0,T;W^{\alpha,2}\left(  \mathbb{R}^{d}\right)  \right)  $ the spaces
$L^{\infty}\left(  0,T;L^{2}\left(  \mathbb{R}^{d}\right)  \right)  $ and
$L^{2}\left(  0,T;W^{\alpha,2}\left(  \mathbb{R}^{d}\right)  \right)  $
endowed respectively with the weak star and weak topology, we have that
$Y_{0}$ is compactly embedded into%
\begin{equation}
Y:=L_{w\ast}^{\infty}\left(  0,T;L^{2}\left(  \mathbb{R}^{d}\right)  \right)
\cap L_{w}^{2}\left(  0,T;W^{\alpha,2}\left(  \mathbb{R}^{d}\right)  \right)
\cap L^{2}\left(  0,T;W_{loc}^{\alpha^{\prime},2}\left(  \mathbb{R}%
^{d}\right)  \right)  . \label{topology of convergence}%
\end{equation}
Notice that%
\[
L^{2}\left(  0,T;W_{loc}^{\alpha^{\prime},2}\left(  \mathbb{R}^{d}\right)
\right)  \subset L^{2}\left(  0,T;C\left(  D\right)  \right)
\]
for every regular bounded domain $D\subset\mathbb{R}^{d}$.

Denote by $\left\{  Q^{N}\right\}  _{N\in\mathbb{N}}$ the laws of $\left\{
h^{N}\right\}  _{N\in\mathbb{N}}$ on $Y_{0}$. From the "boundedness in
probability" of the family $\left\{  Q^{N}\right\}  _{N\in\mathbb{N}}$, in
$Y_{0}$, stated by Lemma \ref{lemma p equal 2} (notice that square
integrability in time of $\left\Vert h_{t}^{N}\right\Vert _{L^{2}\left(
\Omega;W^{\alpha,2}\left(  \mathbb{R}^{d}\right)  \right)  }$ comes from the
assumption $\alpha_{0}-\rho_{0}\leq1$\ which implies $\alpha-\rho_{0}<1$) and
Lemma \ref{lemma fractional Sobolev}, it follows that the family $\left\{
Q^{N}\right\}  _{N\in\mathbb{N}}$ is tight in $Y$, hence relatively compact,
by Prohorov theorem. From every subsequence of $\left\{  Q^{N}\right\}
_{N\in\mathbb{N}}$ it is possible to extract a further subsequence which
converges to a probability measure $Q$ on $Y$. We shall prove that every such
limit measure $Q$ is a Dirac measure $Q=\delta_{u}$ concentrated to the same
element $u\in Y$, hence the whole sequence $\left\{  Q^{N}\right\}
_{N\in\mathbb{N}}$ converges to $\delta_{u}$; and also the processes $\left\{
h^{N}\right\}  _{N\in\mathbb{N}}$ converge in probability to $u$.

Finally, since $\alpha^{\prime}<\alpha$ are arbitrary in the interval
$(d/2,\alpha_{0})$, in Theorem \ref{Thm 1} we have stated the weak convergence
in $L^{2}\left(  0,T;W^{\alpha,2}\left(  \mathbb{R}^{d}\right)  \right)  $ and
the strong convergence in $L^{2}\left(  0,T;W_{loc}^{\alpha,2}\left(
\mathbb{R}^{d}\right)  \right)  $ with the same symbol $\alpha\in
(d/2,\alpha_{0})$.

\subsection{Convergence}

Let us consider also the auxiliary equation%
\begin{equation}
\frac{\partial u}{\partial t}=\Delta u+u\left(  1-u\right)  ^{+},\qquad
u|_{t=0}=u_{0}. \label{FKPP plus}%
\end{equation}

\begin{definition}
\label{def weak sol}Given $u_{0}:\mathbb{R}^{d}\rightarrow\mathbb{R}$
measurable, with $0\leq u_{0}\left(  x\right)  \leq1$ (resp. $u_{0}\left(
x\right)  \geq0$), we call a measurable function $u:\left[  0,T\right]
\times\mathbb{R}^{d}\rightarrow\mathbb{R}$ a weak solution of equation
(\ref{FKPP}) (resp. of equation (\ref{FKPP plus})), if
\[
0\leq u_{t}\left(  x\right)  \leq1
\]
(resp. $u_{t}\left(  x\right)  \geq0$) for a.e. $\left(  t,x\right)
\in\left[  0,T\right]  \times\mathbb{R}^{d}$ and%
\begin{equation}
\left\langle u_{t},\phi\right\rangle =\left\langle u_{0},\phi\right\rangle
+\int_{0}^{t}\left\langle u_{r},\Delta\phi\right\rangle \mathrm{d}r+\int%
_{0}^{t}\left\langle \left(  1-u_{r}\right)  u_{r},\phi\right\rangle
\mathrm{d}r \label{weakFKPP}%
\end{equation}
(resp. with the term $\left(  1-u_{r}\right)  ^{+}$ in place of $\left(
1-u_{r}\right)  $) for all $\phi\in C_{c}^{\infty}\left(  \mathbb{R}%
^{d}\right)  $ and a.e. $t\in\left[  0,T\right]  $.
\end{definition}

\begin{remark}
\label{Remark weak solutions}If $u:\left[  0,T\right]  \times\mathbb{R}%
^{d}\rightarrow\mathbb{R}$ is a measurable function, with $0\leq u_{t}\left(
x\right)  \leq1$ (resp. $u_{t}\left(  x\right)  \geq0$), such that
\[
0=\int_{0}^{T}\int_{\mathbb{R}^{d}}\left(  \frac{\partial\phi_{t}}{\partial
t}+\Delta\phi_{t}+\left(  1-u_{t}\right)  \phi_{t}\right)  u_{t}%
dxdt+\left\langle u_{0},\phi_{0}\right\rangle
\]
(resp. with the term $\left(  1-u_{r}\right)  ^{+}$ in place of $\left(
1-u_{r}\right)  $) for all time-dependent test functions $\phi_{t}$, of class
$C_{c}^{\infty}\left(  \left[  0,T\right]  \times\mathbb{R}^{d}\right)  $,
then one can prove (by taking test functions $\phi_{t}\left(  x\right)  $ of
the form $\eta_{t}^{\epsilon}\cdot\phi\left(  x\right)  $ with $\eta
_{t}^{\epsilon}$ converging to $1_{\cdot\leq t}$) that, for every
time-independent test function $\phi\in C_{c}^{\infty}\left(  \mathbb{R}%
^{d}\right)  $ we have that identity (\ref{weakFKPP}) (resp. with the term
$\left(  1-u_{r}\right)  ^{+}$ in place of $\left(  1-u_{r}\right)  $) holds.
\end{remark}

\begin{lemma}
\label{lemma:support of Q} {Under the assumptions of Theorem \ref{Thm 1}} $Q$
is supported on the set of weak solutions of equation (\ref{FKPP plus}).
\end{lemma}

\begin{proof}
\textbf{Step 1}. We apply remark \ref{Remark weak solutions}. For each
$\phi\in C_{c}^{\infty}\left(  \left[  0,T\right]  \times\mathbb{R}%
^{d}\right)  $, we introduce two functionals%
\[
u\mapsto\Psi_{\phi}\left(  u\right)  :=\int_{0}^{T}\int_{\mathbb{R}^{d}%
}\left(  \frac{\partial\phi_{t}}{\partial t}+\Delta\phi_{t}+\left(
1-u_{t}\right)  ^{+}\phi_{t}\right)  u_{t}dxdt+\left\langle u_{0},\phi
_{0}\right\rangle
\]%
\[
u\mapsto\Psi_{\phi}^{+}\left(  u\right)  :=\int_{0}^{T}\int_{\mathbb{R}^{d}%
}u_{t}\phi_{t}dxdt.
\]
They are continuous on $Y$: since $\phi$ is bounded measurable and compact
support and we have at most the quadratic term $u_{t}^{2}$ under the integral
signs, the topology of $L^{2}\left(  0,T;L_{loc}^{2}\left(  \mathbb{R}%
^{d}\right)  \right)  $, weaker than the topology of $Y$, is sufficient to
prove continuity. Denote by $Q^{N}$ the law of $h_{t}^{N}$ and assume a
subsequence $Q^{N_{k}}$ weakly converges, in the topology of the space $Y$
defined by (\ref{topology of convergence}), to a probability measure $Q$. By
Portmanteau theorem, for every $\epsilon>0$,%
\[
Q\left(  u:\left\vert \Psi_{\phi}\left(  u\right)  \right\vert >\epsilon
\right)  \leq\underset{k\rightarrow\infty}{\lim\inf}\ Q^{N_{k}}\left(
u:\left\vert \Psi_{\phi}\left(  u\right)  \right\vert >\epsilon\right)
=\underset{k\rightarrow\infty}{\lim\inf} \ P\left(  \left\vert \Psi_{\phi
}\left(  h_{\cdot}^{N_{k}}\right)  \right\vert >\epsilon\right)  .
\]
To show $Q\left(  u:\left\vert \Psi_{\phi}\left(  u\right)  \right\vert
>\epsilon\right)  = 0$ we prove in Step 2 below that this lim inf is zero.
Since this holds for every $\epsilon>0$, we deduce $Q\left(  u:\Psi_{\phi
}\left(  u\right)  =0\right)  =1$. By a classical argument of density of a
countable set of test functions, we deduce
\[
Q\left(  \Psi_{\phi}\left(  u\right)  =0\text{ for all }\phi\in C_{c}^{\infty
}\left(  \left[  0,T\right]  \times\mathbb{R}^{d}\right)  \right)  =1.
\]
Similarly, if $\phi_{t}\geq0$, $\phi\in C_{c}^{\infty}\left(  \left[
0,T\right]  \times\mathbb{R}^{d}\right)  $, we apply the same argument to
$\Psi_{\phi}^{+}$ and get
\[
Q\left(  u:\int_{0}^{T}\int_{\mathbb{R}^{d}}u_{t}\phi_{t}dxdt<0\right)
\leq\underset{k\rightarrow\infty}{\lim\inf} \ P\left(  \int_{0}^{T}%
\int_{\mathbb{R}^{d}}h_{t}^{N}\phi_{t}dxdt<0\right)  =0
\]
hence $Q$ is supported on functions $u$ such that $u_{t}\left(  x\right)
\geq0$ for a.e. $\left(  t,x\right)  \in\left[  0,T\right]  \times
\mathbb{R}^{d}$. These two properties prove that $Q$ is supported on the set
of weak solutions of equation (\ref{FKPP plus}).

\textbf{Step 2}. It remains to prove that $\underset{k\rightarrow\infty
}{\lim\inf} \ P\left(  \left\vert \Psi_{\phi}\left(  h_{\cdot}^{N_{k}}\right)
\right\vert >\epsilon\right)  =0$. Let us write $N$ instead of $N_{k}$ for
simplicity of notation. We have%
\[
\Psi_{\phi}\left(  h_{\cdot}^{N}\right)  =\int_{0}^{T}\int_{\mathbb{R}^{d}%
}\left(  \frac{\partial\phi_{t}}{\partial t}+\Delta\phi_{t}+\left(
1-h_{t}^{N}\right)  ^{+}\phi_{t}\right)  h_{t}^{N}dxdt+\left\langle u_{0}%
,\phi_{0}\right\rangle .
\]
By It\^{o} formula, for every $\phi_{t}\in C_{c}^{\infty}\left(  \left[
0,T\right]  \times\mathbb{R}^{d}\right)  $, one has
\begin{align*}
0  &  =\int_{0}^{T}\int_{\mathbb{R}^{d}}\left(  \frac{\partial\phi_{t}%
}{\partial t}+\Delta\phi_{t}\right)  h_{t}^{N}dxdt+\int_{0}^{T}\int%
_{\mathbb{R}^{d}}\theta_{N}\ast\left(  \left(  1-h_{t}^{N}\right)  ^{+}%
S_{t}^{N}\right)  \phi_{t}dxdt\\
&  +\left\langle h_{0}^{N},\phi_{0}\right\rangle {+}\int_{\mathbb{R}^{d}}%
\int_{0}^{T}\phi_{t}dM_{t}^{1,N}dx+\int_{\mathbb{R}^{d}}\int_{0}^{T}\phi
_{t}dM_{t}^{2,N}dx.
\end{align*}
Hence,
\begin{align*}
\Psi_{\phi}\left(  h_{\cdot}^{N}\right)   &  =\int_{0}^{T}\int_{\mathbb{R}%
^{d}}\left[  \left(  1-h_{t}^{N}\right)  ^{+}h_{t}^{N}-\theta_{N}\ast\left(
\left(  1-h_{t}^{N}\right)  ^{+}S_{t}^{N}\right)  \right]  \phi_{t}dxdt\\
&  {-}\int_{\mathbb{R}^{d}}\int_{0}^{T}\phi_{t}dM_{t}^{1,N}dx{-}%
\int_{\mathbb{R}^{d}}\int_{0}^{T}\phi_{t}dM_{t}^{2,N}dx\\
&  {-}\left\langle h_{0}^{N},\phi_{0}\right\rangle {+}\left\langle u_{0}%
,\phi_{0}\right\rangle .
\end{align*}
In order to prove $\lim_{N\rightarrow\infty}P\left(  \left\vert \Psi_{\phi
}\left(  h_{\cdot}^{N}\right)  \right\vert >\varepsilon\right)  =0$, it is
sufficient to prove the same result for each one of the previous terms.
Lemma \ref{Lemma 2 auxiliary} deals with the first term and the two martingale
terms can be treated by Chebyshev's inequality and Lemma
\ref{Lemma 3 auxiliary} below. The terms
\[
{-}\left\langle h_{0}^{N},\phi_{0}\right\rangle {+}\left\langle u_{0},\phi
_{0}\right\rangle ={-}\left\langle S_{0}^{N},\theta_{N}\left(  -\cdot\right)
\ast\phi_{0}\right\rangle {+}\left\langle u_{0},\phi_{0}\right\rangle
\]
converges to zero in probability by the assumption that $S_{0}^{N}$ converges
weakly to $u_{0}\left(  x\right)  dx$, as $N\rightarrow\infty$, in
probability.

\end{proof}

\begin{lemma}
\label{Lemma 2 auxiliary} It holds
\begin{align*}
\int_{0}^{T}\int_{\mathbb{R}^{d}}\left[  \left(  1-h_{t}^{N}\right)  ^{+}%
h_{t}^{N}-\theta_{N}\ast\left(  \left(  1-h_{t}^{N}\right)  ^{+}S_{t}%
^{N}\right)  \right]  \phi_{t}dxdt \rightarrow0 \quad\text{as } N\to\infty
\end{align*}
in probability.
\end{lemma}

\begin{proof}
We split the inner integral into
\begin{align*}
&  \left\vert \left\langle \theta_{N} \ast( S_{t}^{N} \left(  1-h_{t}%
^{N}\right)  ^{+}) - h_{t}^{N} \left(  1-h_{t}^{N}\right)  ^{+}, \phi_{t}
\right\rangle \right\vert \\
&  \leq\left\vert \left\langle \theta_{N} \ast( S_{t}^{N} \left(  1-h_{t}%
^{N}\right)  ^{+}) - S_{t}^{N} \left(  1-h_{t}^{N}\right)  ^{+}, \phi_{t}
\right\rangle \right\vert \\
&  \hspace*{12pt} + \left\vert \left\langle S_{t}^{N} \left(  1-h_{t}%
^{N}\right)  ^{+} - S_{t}^{N} \left(  1-h_{t}\right)  ^{+}, \phi
_{t}\right\rangle \right\vert \\
&  \hspace*{12pt} +\left\vert \left\langle S_{t}^{N} \left(  1-h_{t}\right)
^{+} - h_{t}^{N} \left(  1-h_{t}\right)  ^{+}, \phi_{t} \right\rangle
\right\vert \\
&  \hspace*{12pt} + \left\vert \left\langle h_{t}^{N} \left(  1-h_{t}\right)
^{+} - h_{t}^{N} \left(  1-h_{t}^{N}\right)  ^{+}, \phi_{t} \right\rangle
\right\vert \\
&  = I_{t}^{N}+II_{t}^{N}+III_{t}^{N}+IV_{t}^{N},
\end{align*}
where $h$ denotes the almost sure limit of $(h^{N})_{N\in\mathbb{N}}$ given by
Skorokhod's representation theorem. To prove Lemma \ref{Lemma 2 auxiliary} it
is sufficient to show that each term on the right-hand side integrated in time
converges in probability to zero. In order to prove that, it is sufficient to
show that the expectation converge to zero for every $t\in[0,T]$, because
\[
P\left(  \int_{0}^{T} I_{t}^{N} \mathrm{d}t > \varepsilon\right)  \leq\frac
{1}{\varepsilon} E \int_{0}^{T} I_{t}^{N} \mathrm{d}t = \frac{1}{\varepsilon}
\int_{0}^{T} E I_{t}^{N} \mathrm{d}t \to0.
\]
Note, there is a compact set $K$, such that $K\supset\cup_{t\in[0,T]}%
\operatorname{supp}(\phi_{t})$. For ease of notation we omit the subscript $t$
in the following.\newline First,
\begin{align*}
I^{N}  &  =\left\vert \left\langle \theta_{N} \ast( S_{t}^{N} \left(
1-h_{t}^{N}\right)  ^{+}) - S_{t}^{N} \left(  1-h_{t}^{N}\right)  ^{+},
\varphi\right\rangle \right\vert = \left\vert \left\langle S_{t}^{N} \left(
1-h_{t}^{N}\right)  ^{+}, \theta_{N}\ast\varphi- \varphi\right\rangle
\right\vert \\
&  \leq\left[  S_{t}^{N}\right]  \left\Vert \theta_{N}\ast\varphi-
\varphi\right\Vert _{\infty}.
\end{align*}
Hence,
\[
E I^{N} \leq E \left[  S_{t}^{N}\right]  \left\Vert \theta_{N}\ast\varphi-
\varphi\right\Vert _{\infty} \leq\underbrace{\left\Vert \theta_{N}\ast\varphi-
\varphi\right\Vert _{\infty}}_{\to0} \underbrace{\sup_{N\in\mathbb{N}} E
\left[  S_{T}^{N}\right]  }_{<\infty}.
\]
Second, observe that
\begin{align*}
II^{N}  &  =\left\vert \left\langle S_{t}^{N} \left(  1-h_{t}^{N}\right)  ^{+}
- S_{t}^{N} \left(  1-h_{t}\right)  ^{+} , \varphi\right\rangle \right\vert
\leq|\varphi|_{\infty}\left[  S_{T}^{N}\right]  \sup_{x\in K} \left\vert
\left(  1-h_{t}^{N}(x)\right)  ^{+} - \left(  1-h_{t}(x)\right)  ^{+}
\right\vert \\
&  \leq|\varphi|_{\infty}\left[  S_{T}^{N}\right]  \sup_{x\in K} \left\vert
h_{t}^{N}(x)-h_{t}(x) \right\vert .
\end{align*}
and by Sobolev embedding and Lemma \ref{lemma p equal 2} we have
\[
\sup_{x\in K} \left\vert h_{t}^{N}(x)-h_{t}(x) \right\vert \rightarrow0 .
\]
It follows
\begin{align*}
E II^{N}  &  \leq\left\Vert \varphi\right\Vert _{\infty}E\left(  \left[
S_{T}^{N}\right]  \sup_{x\in K} \left\vert h_{t}^{N}(x)-h_{t}(x) \right\vert
\right) \\
&  \leq\left\Vert \varphi\right\Vert _{\infty}E\left(  \left[  S_{T}%
^{N}\right]  ^{2}\right)  E\left(  \sup_{x\in K} \left\vert h_{t}^{N}%
(x)-h_{t}(x) \right\vert ^{2}\right) \\
&  \leq\left\Vert \varphi\right\Vert _{\infty}\underbrace{\sup_{N
\in\mathbb{N}}E\left(  \left[  S_{T}^{N}\right]  ^{2}\right)  }_{<\infty}
\underbrace{ E\left(  \sup_{x\in K} \left\vert h_{t}^{N}(x)-h_{t}(x)
\right\vert ^{2}\right)  }_{\to0} \to0.
\end{align*}
The third term converges to zero pointwise due to the weak convergence of
$S^{N}$ and $h^{N}$.\newline Finally, the last term also converges pointwise.
From Section \ref{subsect compactness} we have
\[
\int_{K} \left\vert h_{t}^{N}(x)-h_{t}(x) \right\vert ^{2} \mathrm{d}%
x\rightarrow0 .
\]
Therefore,
\begin{align*}
&  \left\vert \left\langle h_{t}^{N} \left(  1-h_{t}^{N}\right)  ^{+} -
h_{t}^{N} \left(  1-h_{t}\right)  ^{+} , \varphi\right\rangle \right\vert \\
&  \leq\left(  \int_{K} \left\vert h_{t}^{N} \right\vert ^{2} \mathrm{d}x
\right)  ^{\frac{1}{2}} \left(  \int_{K} \left\vert \left(  1-h_{t}%
^{N}\right)  ^{+} - \left(  1-h_{t}\right)  ^{+} \right\vert ^{2} \mathrm{d}x
\right)  ^{\frac{1}{2}}\\
&  \leq\underbrace{ \left(  \int_{K} \left\vert h_{t}^{N} \right\vert ^{2}
\mathrm{d}x\right)  ^{\frac{1}{2}}}_{\to\left(  \int_{K} \left\vert h_{t}
\right\vert ^{2} \mathrm{d}x \right)  ^{\frac{1}{2}} < \infty} \left(
\int_{K} \left\vert h_{t}^{N} - h_{t}\right\vert ^{2} \mathrm{d}x \right)
^{\frac{1}{2}} \to0.
\end{align*}

\end{proof}

In the next lemma we denote by $C$ any constant depending only on $T$,
$\left\Vert \theta\right\Vert _{L^{2}}^{2}$, $\sup_{N}\varepsilon_{N}^{-d}/N$,
$\left\Vert \phi\right\Vert _{\infty}$, $\left\Vert \nabla\phi\right\Vert
_{\infty}$, $E\left[  \left[  S_{T}^{N}\right]  \right]  $.

\begin{lemma}
\label{Lemma 3 auxiliary}For $i=1,2$%
\[
E\left[  \left\vert \int_{\mathbb{R}^{d}}\int_{0}^{T}\phi_{t}(x)\mathrm{d}%
M_{t}^{i,N}(x)\mathrm{d}x\right\vert ^{2}\right]  \leq C N^{\beta-1}.
\]

\end{lemma}

\begin{proof}
Set
\[
g_{t}^{N}(y):=-\int_{\mathbb{R}^{d}}\phi_{t}(x)\nabla\theta_{N}\left(
x-y\right)  dx=\int_{\mathbb{R}^{d}}\nabla\phi_{t}(x)\theta_{N}\left(
x-y\right)  dx.
\]
For the first martingale term we have
\begin{align*}
E\left[  \left\vert \int_{\mathbb{R}^{d}}\int_{0}^{T}\phi_{t}(x)\mathrm{d}%
M_{t}^{1,N}(x)\mathrm{d}x\right\vert ^{2}\right]   &  =\frac{2}{N^{2}}%
\sum_{a\in A^{N}}E\left[  \int_{0}^{T}1_{t\in I^{a}}\left\vert g_{t}%
^{N}\left(  X_{t}^{a}\right)  \right\vert ^{2}\mathrm{d}t\right] \\
&  \leq\frac{2}{N}\left\Vert \theta_{N}\right\Vert _{L^{2}}^{2}\left\Vert
\nabla\phi\right\Vert _{\infty}^{2}E\int_{0}^{T}\left[  S_{t}^{N}\right]  dt.
\end{align*}
The assertion for $i=1$ follows from Lemma \ref{lemma total mass} and
\[
\frac{1}{N}\left\Vert \theta_{N}\right\Vert _{L^{2}}^{2} \leq C \frac
{\epsilon_{N}^{-d}}{N} = C N^{\beta-1}.
\]
Set
\[
\widetilde{g}_{t}^{N}(y):=-\int_{\mathbb{R}^{d}}\phi_{t}(x)\theta_{N}\left(
x-y\right)  dx,
\]
then for the second martingale term we have%
\begin{align*}
E\left[  \left\vert \int_{\mathbb{R}^{d}}\int_{0}^{T}\phi_{t}(x)\mathrm{d}%
M_{t}^{2,N}(x)\mathrm{d}x\right\vert ^{2}\right]   &  =\frac{1}{N^{2}}%
\sum_{a\in A^{N}}E\left[  \int_{0}^{T}1_{t\in I^{a}}\left\vert \widetilde{g}%
_{t}^{N}\left(  X_{t}^{a}\right)  \right\vert ^{2}\lambda_{t}^{a}%
\mathrm{d}t\right] \\
&  \leq\frac{1}{N}\left\Vert \theta_{N}\right\Vert _{L^{2}}^{2}\left\Vert
\phi\right\Vert _{L^{\infty}}^{2}E\int_{0}^{T}\left[  S_{t}^{N}\right]  dt
\end{align*}
and we conclude by the same argument. This completes the proof.
\end{proof}

\subsection{Auxiliary results\label{section auxiliary}}

\begin{theorem}
\label{Thm PDE}There is at most one weak solution of equation (\ref{FKPP plus}%
). The unique solution has the additional property $u_{t}\left(  x\right)
\leq1$, hence it is also the unique solution of (\ref{FKPP}).
\end{theorem}

\begin{proof}
Let $u^{1},u^{2}$ be two weak solutions of the equation (\ref{FKPP plus}) with
the same initial condition $u_{0}$. Let $\{\rho_{\varepsilon}%
(x)\}_{\varepsilon}$ be a family of standard symmetric mollifiers. For any
$\varepsilon>0$ and $x\in\mathbb{R}^{d}$ we can use $\rho_{\varepsilon
}(x-\cdot)$ as test function in the equation (\ref{weakFKPP}). Set
$u_{\varepsilon}^{i}(t,x)=u^{i}(t,x)\ast_{x}\rho_{\varepsilon}(x)$ for
$i=1,2$. Then we have%
\[
u_{\varepsilon}^{i}(t,x)=(u_{0}\ast\rho_{\varepsilon})(x)+\int_{0}^{t}\Delta
u_{\varepsilon}^{i}(s,x)\,ds+\int_{0}^{t}(\rho_{\varepsilon}\ast(1-u^{i}%
)^{+}u^{i})(s,x)\,ds.
\]
Writing this identity in mild form we obtain (we write $u^{i}\left(  t\right)
$ for the function $u^{i}\left(  s,\cdot\right)  $ and $S(t)$ for $e^{tA}$)%
\[
u_{\varepsilon}^{i}(t)=S(t)(u_{0}\ast\rho_{\varepsilon})+\int_{0}%
^{t}S(t-s)\left(  \rho_{\varepsilon}\ast\left(  (1-u^{i}\left(  s\right)
)^{+}u^{i}\left(  s\right)  \right)  \right)  ds.
\]
Write $g\left(  u\right)  $ for the function $u\rightarrow(1-u)^{+}u$ from
$[0,\infty)$ into $[0,\infty)$. The function $U=u^{1}-u^{2}$ satisfies%
\[
\rho_{\varepsilon}\ast U(t)=\int_{0}^{t}S(t-s)\left(  \rho_{\varepsilon}%
\ast\left[  g\left(  u^{1}\left(  s\right)  \right)  -g\left(  u^{2}\left(
s\right)  \right)  \right]  \right)  \,ds.
\]
Taking the limit as $\varepsilon\rightarrow0$ we have%
\[
U(t)=\int_{0}^{t}S(t-s)\left[  g\left(  u^{1}\left(  s\right)  \right)
-g\left(  u^{2}\left(  s\right)  \right)  \right]  \,ds.
\]
Hence,
\[
\Vert U(t)\Vert_{\infty}\leq\int_{0}^{t}\Vert g\left(  u^{1}\left(  s\right)
\right)  -g\left(  u^{2}\left(  s\right)  \right)  \Vert_{\infty}ds.
\]
Notice that the function $g$ is globally Lipschitz, with Lipschitz constant 1
(compute the derivative). It follows
\[
\Vert U(t)\Vert_{\infty}\leq\int_{0}^{t}\Vert U(s)\Vert_{\infty}ds.
\]
By Gronwall's lemma we conclude $U=0$.

It is a classical result that equation (\ref{FKPP}) has a unique weak
solution, with the property $u_{t}\in\left[  0,1\right]  $, being $u_{0}$
bounded, uniformly continuous and of class $L^{2}$ (see \cite{Smoller},
Chapter 14, Section A). Hence, this solution is also a solution of equation
(\ref{FKPP plus}) and coincides with the unique weak solution of that equation.
\end{proof}

\begin{lemma}
\label{lemma total mass} There exists a $\gamma>0$ such that $\sup_{N}E\left[
e^{\gamma\left[  S_{T}^{N}\right]  }\right]  <\infty$.
\end{lemma}

This lemma follows from the boundedness of the rates $\lambda^{a,N}_{t}$.
Indeed, this boundedness implies that the process $t\mapsto\left[  S_{t}%
^{N}\right]  $ is stochastically dominated by $\frac{1}{N}Y_{N\left[
S_{0}^{N}\right]  }(\cdot)$, where $Y_{k}$ is a Yule process with birth rate
$1$ and $Y_{k}(0)=k$, see also \cite{FlaLeim}.

The following proposition gives an easy sufficient condition for assumption
(\ref{initial cond}) on the initial condition.

\begin{proposition}
\label{Proposition sufficient cond}Assume that $X_{0}^{i}$, $i=1,...,N$, are
independent identically distributed r.v with common probability density
$u_{0}\in W^{\rho_{0},2}\left(  \mathbb{R}^{d}\right)  $, that assumption
(\ref{assumption on theta}) holds and that $\rho_{0}\leq\alpha_{0}$. Then
\[
\sup_{N\in\mathbb{N}}E\left\Vert \theta_{N}\ast S_{0}^{N}\right\Vert
_{W^{\rho,2}\left(  \mathbb{R}^{d}\right)  }^{2}<\infty.
\]

\end{proposition}

\begin{proof}
\textbf{Step 1}. To clarify the proof below, for pedagogical reasons we first
treat the case $\rho_{0}=0$. By the i.i.d. property%
\begin{align*}
E\int_{\mathbb{R}^{d}}\left\vert \left(  \theta_{N}\ast S_{0}^{N}\right)
\left(  x\right)  \right\vert ^{2}dx  &  =\frac{1}{N^{2}}\int_{\mathbb{R}^{d}%
}E\left[  \left(  \sum_{i=1}^{N}\theta_{N}\left(  x-X_{0}^{i}\right)  \right)
^{2}\right]  dx\\
&  =\frac{1}{N}\int_{\mathbb{R}^{d}}E\left[  \left\vert \theta_{N}\left(
x-X_{0}^{1}\right)  \right\vert ^{2}\right]  dx+\frac{N\left(  N-1\right)
}{N^{2}}\int_{\mathbb{R}^{d}}E\left[  \theta_{N}\left(  x-X_{0}^{1}\right)
\right]  ^{2}dx.
\end{align*}
For the last term notice that
\[
E\left[  \theta_{N}\left(  x-X_{0}^{1}\right)  \right]  =\left(  \theta
_{N}\ast u_{0}\right)  \left(  x\right)
\]
hence,
\[
\frac{N\left(  N-1\right)  }{N^{2}}\int_{\mathbb{R}^{d}}E\left[  \theta
_{N}\left(  x-X_{0}^{1}\right)  \right]  ^{2}dx\leq\left\Vert \theta_{N}\ast
u_{0}\right\Vert _{L^{2}}^{2}\leq C,
\]
because $\theta_{N}\ast u_{0}\rightarrow u_{0}$ in $L^{2}\left(
\mathbb{R}^{d}\right)  $. About the first term, we have%
\[
\frac{1}{N}\int_{\mathbb{R}^{d}}E\left[  \left\vert \theta_{N}\left(
x-X_{0}^{1}\right)  \right\vert ^{2}\right]  dx=\frac{1}{N}E\left[
\int_{\mathbb{R}^{d}}\left\vert \theta_{N}\left(  x-X_{0}^{1}\right)
\right\vert ^{2}dx\right]  =\frac{1}{N}E\left[  \int_{\mathbb{R}^{d}%
}\left\vert \theta_{N}\left(  x\right)  \right\vert ^{2}dx\right]  \leq C
\]
by (\ref{bound on theta N}). Hence $E\int_{\mathbb{R}^{d}}\left\vert \left(
\theta_{N}\ast S_{0}^{N}\right)  \left(  x\right)  \right\vert ^{2}dx\leq C$.

If $\rho_{0}$ is an integer, the proof can be easily modified. Let us treat
the general case in the next step.

\textbf{Step 2}. Similarly to a property already used in the proof of Lemma
\ref{lemma martingale 1}, one has the following translation invariance
property:
\[
\left(  \left(  I-A\right)  ^{\rho_{0}/2}\theta_{N}\left(  \cdot-X_{0}%
^{i}\right)  \right)  \left(  x\right)  =\left(  \left(  I-A\right)
^{\rho_{0}/2}\theta_{N}\right)  \left(  x-X_{0}^{i}\right)  .
\]
Therefore%
\begin{align*}
E\int_{\mathbb{R}^{d}}\left\vert \left(  \left(  I-A\right)  ^{\rho_{0}%
/2}\left(  \theta_{N}\ast S_{0}^{N}\right)  \right)  \left(  x\right)
\right\vert ^{2}dx  &  =\frac{1}{N^{2}}\int_{\mathbb{R}^{d}}E\left[  \left(
\sum_{i=1}^{N}\left(  \left(  I-A\right)  ^{\rho_{0}/2}\theta_{N}\left(
\cdot-X_{0}^{i}\right)  \right)  \left(  x\right)  \right)  ^{2}\right]  dx\\
&  =\frac{1}{N^{2}}\int_{\mathbb{R}^{d}}E\left[  \left(  \sum_{i=1}^{N}\left(
\left(  I-A\right)  ^{\rho_{0}/2}\theta_{N}\right)  \left(  x-X_{0}%
^{i}\right)  \right)  ^{2}\right]  dx\\
&  =\frac{1}{N}\int_{\mathbb{R}^{d}}E\left[  \left\vert \left(  \left(
I-A\right)  ^{\rho_{0}/2}\theta_{N}\right)  \left(  x-X_{0}^{1}\right)
\right\vert ^{2}\right]  dx\\
&  +\frac{N\left(  N-1\right)  }{N^{2}}\int_{\mathbb{R}^{d}}E\left[  \left(
\left(  I-A\right)  ^{\rho_{0}/2}\theta_{N}\right)  \left(  x-X_{0}%
^{1}\right)  \right]  ^{2}dx.
\end{align*}
For the last term we have (using the fact that the operator $\left(
I-A\right)  ^{\rho_{0}/2}$ is self-adjoint in $L^{2}\left(  \mathbb{R}%
^{d}\right)  $)
\begin{align*}
E\left[  \left(  \left(  I-A\right)  ^{\rho_{0}/2}\theta_{N}\right)  \left(
x-X_{0}^{1}\right)  \right]   &  =\int\left(  \left(  I-A\right)  ^{\rho
_{0}/2}\theta_{N}\right)  \left(  x-y\right)  u_{0}\left(  y\right)  dy\\
&  =\left\langle \left(  I-A\right)  ^{\rho_{0}/2}\theta_{N},u_{0}\left(
x-\cdot\right)  \right\rangle \\
&  =\left\langle \theta_{N},\left(  I-A\right)  ^{\rho_{0}/2}u_{0}\left(
x-\cdot\right)  \right\rangle \\
&  =\int\theta_{N}\left(  z\right)  \left(  \left(  I-A\right)  ^{\rho_{0}%
/2}u_{0}\left(  x-\cdot\right)  \right)  \left(  z\right)  dz\\
&  =\int\theta_{N}\left(  z\right)  \left(  \left(  I-A\right)  ^{\rho_{0}%
/2}u_{0}\right)  \left(  x-z\right)  dz\\
&  =\left(  \theta_{N}\ast\left(  I-A\right)  ^{\rho_{0}/2}u_{0}\right)
\left(  x\right)
\end{align*}
where we have used again a translation invariance property. Hence,
\begin{align*}
&  \frac{N\left(  N-1\right)  }{N^{2}}\int_{\mathbb{R}^{d}}E\left[  \left(
\left(  I-A\right)  ^{\rho_{0}/2}\theta_{N}\right)  \left(  x-X_{0}%
^{1}\right)  \right]  ^{2}dx\\
&  \leq\left\Vert \theta_{N}\ast\left(  I-A\right)  ^{\rho_{0}/2}%
u_{0}\right\Vert _{L^{2}}^{2}\leq C\left\Vert \left(  I-A\right)  ^{\rho
_{0}/2}u_{0}\right\Vert _{L^{2}}^{2}\leq C
\end{align*}
because the convolutions with $\theta_{N}$ are equibounded in $L^{2}\left(
\mathbb{R}^{d}\right)  $. For the first term, we have%
\begin{align*}
\frac{1}{N}\int_{\mathbb{R}^{d}}E\left[  \left\vert \left(  \left(
I-A\right)  ^{\rho_{0}/2}\theta_{N}\right)  \left(  x-X_{0}^{1}\right)
\right\vert ^{2}\right]  dx  &  =\frac{1}{N}E\left[  \int_{\mathbb{R}^{d}%
}\left\vert \left(  \left(  I-A\right)  ^{\rho_{0}/2}\theta_{N}\right)
\left(  x-X_{0}^{1}\right)  \right\vert ^{2}dx\right] \\
&  =\frac{1}{N}E\left[  \int_{\mathbb{R}^{d}}\left\vert \left(  \left(
I-A\right)  ^{\rho_{0}/2}\theta_{N}\right)  \left(  x\right)  \right\vert
^{2}dx\right] \\
&  \leq\frac{1}{N}C\epsilon_{N}^{-2\rho_{0}}\epsilon_{N}^{-d}\left\Vert
\theta\right\Vert _{W^{\rho_{0},2}\left(  \mathbb{R}^{d}\right)  }^{2}%
=\frac{1}{N}CN^{\frac{\beta}{d}\left(  2\rho_{0}+d\right)  }\\
&  \leq\frac{1}{N}CN^{\frac{\beta}{d}\left(  2\alpha_{0}+d\right)  }\leq
\frac{1}{N}CN^{\frac{\beta}{d}\left(  2\frac{d(1-\beta)}{2\beta}+d\right)
}\leq C
\end{align*}
where we have used Lemma \ref{lemma interpolation} below, $\epsilon
_{N}=N^{\frac{\beta}{d}}$, $\rho_{0}\leq\alpha_{0}$ and the condition
$\alpha_{0}\leq\frac{d(1-\beta)}{2\beta}$ imposed in assumption
(\ref{assumption on theta}). Hence $E\int_{\mathbb{R}^{d}}\left\vert \left(
\left(  I-A\right)  ^{\rho_{0}/2}\left(  \theta_{N}\ast S_{0}^{N}\right)
\right)  \left(  x\right)  \right\vert ^{2}dx\leq C$.
\end{proof}

\begin{lemma}
\label{lemma interpolation}For every $\alpha\geq0$ and $\theta\in W^{\alpha
,2}\left(  \mathbb{R}^{d}\right)  $ there exists a constant $C\geq0$ such
that
\[
\left\Vert \theta_{N}\right\Vert _{W^{\alpha,2}\left(  \mathbb{R}^{d}\right)
}\leq C\epsilon_{N}^{-\alpha}\epsilon_{N}^{-d/2}\left\Vert \theta\right\Vert
_{W^{\alpha,2}\left(  \mathbb{R}^{d}\right)  }.
\]

\end{lemma}

\begin{proof}
First, we compute the Fourier transform of ${\theta_{N}}$, i.e. for all
$\lambda\in\mathbb{R}^{d}$
\begin{align*}
\widehat{\theta_{N}}\left(  \lambda\right)   &  =\int_{\mathbb{R}^{d}%
}e^{i\lambda\cdot x}\theta_{N}\left(  x\right)  dx=\epsilon_{N}^{-d}%
\int_{\mathbb{R}^{d}}e^{i\lambda\cdot x}\theta\left(  \epsilon_{N}%
^{-1}x\right)  dx\\
&  =\int_{\mathbb{R}^{d}}e^{i\epsilon_{N}\lambda\cdot y}\theta\left(
y\right)  dy=\widehat{\theta}\left(  \epsilon_{N}\lambda\right)  .
\end{align*}
Second, note that the norms
\[
f\mapsto\left\Vert f\right\Vert _{W^{\alpha,2}\left(  \mathbb{R}^{d}\right)
}^{2}\text{ and }f\mapsto\int_{\mathbb{R}^{d}}\left(  1+\left\vert
\lambda\right\vert ^{2}\right)  ^{\alpha}\left\vert \widehat{f}\left(
\lambda\right)  \right\vert ^{2}d\lambda
\]
are equivalent. Therefore, there is a constant $C\geq0$, which may change from
instance to instance, such that
\begin{align*}
\left\Vert \theta_{N}\right\Vert _{W^{\alpha,2}\left(  \mathbb{R}^{d}\right)
}^{2}  &  \leq C\int_{\mathbb{R}^{d}}\left(  1+\left\vert \lambda\right\vert
^{2}\right)  ^{\alpha}\left\vert \widehat{\theta_{N}}\left(  \lambda\right)
\right\vert ^{2}d\lambda=C\int_{\mathbb{R}^{d}}\left(  1+\left\vert
\lambda\right\vert ^{2}\right)  ^{\alpha}\left\vert \widehat{\theta}\left(
\epsilon_{N}\lambda\right)  \right\vert ^{2}d\lambda\\
&  =C\epsilon_{N}^{-d}\int_{\mathbb{R}^{d}}\left(  1+\left\vert \epsilon
_{N}^{-1}\eta\right\vert ^{2}\right)  ^{\alpha}\left\vert \widehat{\theta
}\left(  \eta\right)  \right\vert ^{2}d\eta\\
&  =C\epsilon_{N}^{-d}\epsilon_{N}^{-2\alpha}\int_{\mathbb{R}^{d}}\left(
\epsilon_{N}^{2}+\left\vert \eta\right\vert ^{2}\right)  ^{\alpha}\left\vert
\widehat{\theta}\left(  \eta\right)  \right\vert ^{2}d\eta\\
&  \overset{\alpha\geq0}{\leq}C\epsilon_{N}^{-d}\epsilon_{N}^{-2\alpha}%
\int_{\mathbb{R}^{d}}\left(  1+\left\vert \eta\right\vert ^{2}\right)
^{\alpha}\left\vert \widehat{\theta}\left(  \eta\right)  \right\vert ^{2}%
d\eta\\
&  \leq C\epsilon_{N}^{-d}\epsilon_{N}^{-2\alpha}\left\Vert \theta\right\Vert
_{W^{\alpha,2}\left(  \mathbb{R}^{d}\right)  }^{2}.
\end{align*}

\end{proof}


\begin{lemma}
\label{lemma positive}The linear bounded operator $f\mapsto\left(  I-A\right)
^{\epsilon/2}e^{tA}f$ on $L^{2}\left(  \mathbb{R}^{d}\right)  $ is positive,
i.e. $f\geq0$ implies $\left(  I-A\right)  ^{\epsilon/2}e^{tA}f\geq0$.
\end{lemma}


\begin{proof}
Let $f$ be a non negative function of class $L^{2}\left(  \mathbb{R}%
^{d}\right)  $. In order to prove that the function $g:=\left(  I-A\right)
^{\epsilon/2}e^{tA}f$ is non negative, it is sufficient to prove that its
Fourier transform $\widehat{g}$ is non negative definite, namely
$\operatorname{Re}\sum_{i,j=1}^{n}\widehat{g}\left(  \lambda_{i}-\lambda
_{j}\right)  \xi_{i}\overline{\xi}_{j}\geq0$ for every $n\in\mathbb{N}$,
$\lambda_{i}\in\mathbb{R}^{d}$ and $\xi_{i}\in\mathbb{C}$, $i=1,...,n$. We
have
\[
\widehat{g}\left(  \lambda\right)  =\left(  1+\left\vert \lambda\right\vert
^{2}\right)  ^{\epsilon/2}e^{-t\left\vert \lambda\right\vert ^{2}}%
\widehat{f}\left(  \lambda\right)
\]
and thus we have to prove that, given $n\in\mathbb{N}$, $\lambda_{i}%
\in\mathbb{R}^{d}$ and $\xi_{i}\in\mathbb{C}$, $i=1,...,n$, one has
\[
\operatorname{Re}\sum_{i,j=1}^{n}\left(  1+\left\vert \lambda_{i}-\lambda
_{j}\right\vert ^{2}\right)  ^{\epsilon/2}e^{-t\left\vert \lambda_{i}%
-\lambda_{j}\right\vert ^{2}}\widehat{f}\left(  \lambda_{i}-\lambda
_{j}\right)  \xi_{i}\overline{\xi}_{j}\geq0
\]
namely%
\[
\sum_{i=1}^{n}\operatorname{Re}\widehat{f}\left(  0\right)  \xi_{i}%
\overline{\xi}_{j}+\sum_{i<j}\left(  1+\left\vert \lambda_{i}-\lambda
_{j}\right\vert ^{2}\right)  ^{\epsilon/2}e^{-t\left\vert \lambda_{i}%
-\lambda_{j}\right\vert ^{2}}\left(  \operatorname{Re}\widehat{f}\left(
\lambda_{i}-\lambda_{j}\right)  \xi_{i}\overline{\xi}_{j}+\operatorname{Re}%
\widehat{f}\left(  \lambda_{j}-\lambda_{i}\right)  \xi_{j}\overline{\xi}%
_{i}\right)  \geq0.
\]
Corresponding to any couple $\left(  i,j\right)  \in\left\{  1,...,n\right\}
^{2}$, let $\widetilde{\xi}_{1},...,\widetilde{\xi}_{n}\in\mathbb{C}$ be such
that $\widetilde{\xi}_{i}=\xi_{i}$, $\widetilde{\xi}_{j}=\xi_{j}$, and
$\widetilde{\xi}_{k}=0$ for $k\notin\left\{  i,j\right\}  $. We know that
$\sum_{ij=1}^{n}\operatorname{Re}\widehat{f}\left(  \lambda_{i}-\lambda
_{j}\right)  \widetilde{\xi}_{i}\overline{\widetilde{\xi}}_{j}\geq0$, hence,
if $i=j$%
\[
\operatorname{Re}\widehat{f}\left(  0\right)  \xi_{i}\overline{\xi}_{i}\geq0
\]
while for $i\neq j$
\[
\operatorname{Re}\widehat{f}\left(  \lambda_{i}-\lambda_{j}\right)  \xi
_{i}\overline{\xi}_{j}+\operatorname{Re}\widehat{f}\left(  \lambda_{j}%
-\lambda_{i}\right)  \xi_{j}\overline{\xi}_{i}\geq0.
\]
Using these two facts above we get the result.
\end{proof}

\section{Appendix\label{appendix B}}

Since we deduced the threshold $\beta<1/2$ from an approach based on Sobolev's
embedding theorem in the spaces $W^{\alpha,2}$, it is natural to ask what
happens if we use $W^{\alpha,p}$-topologies (which allow one to use much
smaller $\alpha$, taking advantage of large $p$) or even H\"{o}lder
topologies. We have done partial computations in these directions and the
threshold $\beta<1/2$ is the same in all approaches we have outlined. Let us
show here just a partial computation in H\"{o}lder norms.

The restriction (in all approaches) seems to come from the estimate of the
Brownian martingale. Recall it is given by%

\[
\widetilde{M}_{t}^{1,N}\left(  x\right)  :=\frac{1}{N}\sum_{a\in\Lambda^{N}%
}\int_{0}^{t}\left(  e^{\left(  t-s\right)  A}\nabla\theta_{N}\right)  \left(
x-X_{s}^{a}\right)  1_{s\in I^{a}}\cdot dB_{s}^{a}.
\]
In order to investigate its H\"{o}lder properties, let us invoke Kolmogorov
regularity criterion. Hence, we estimate, by the Burkh\"{o}lder-Davis-Gundy
inequality,%
\begin{align*}
&  E\left[  \left\vert \widetilde{M}_{t}^{1,N}\left(  x\right)  -\widetilde{M}%
_{t}^{1,N}\left(  x^{\prime}\right)  \right\vert ^{p}\right] \\
&  =\frac{1}{N^{p}}E\left[  \left\vert \sum_{a\in\Lambda^{N}}\int_{0}%
^{t}\left(  \left(  e^{\left(  t-s\right)  A}\nabla\theta_{N}\right)  \left(
x-X_{s}^{a}\right)  -\left(  e^{\left(  t-s\right)  A}\nabla\theta_{N}\right)
\left(  x^{\prime}-X_{s}^{a}\right)  \right)  1_{s\in I^{a}}\cdot dB_{s}%
^{a}\right\vert ^{p}\right] \\
&  \leq\frac{C}{N^{p}}E\left[  \left\vert \sum_{a\in\Lambda^{N}}\int_{0}%
^{t}\left\vert \left(  e^{\left(  t-s\right)  A}\nabla\theta_{N}\right)
\left(  x-X_{s}^{a}\right)  -\left(  e^{\left(  t-s\right)  A}\nabla\theta
_{N}\right)  \left(  x^{\prime}-X_{s}^{a}\right)  \right\vert ^{2}1_{s\in
I^{a}}ds\right\vert ^{p/2}\right] \\
&  =\frac{C}{N^{p}}E\left[  \left\vert \int_{0}^{t}\sum_{a\in\Lambda_{s}^{N}%
}\left\vert \left(  e^{\left(  t-s\right)  A}\nabla\theta_{N}\right)  \left(
x-X_{s}^{a}\right)  -\left(  e^{\left(  t-s\right)  A}\nabla\theta_{N}\right)
\left(  x^{\prime}-X_{s}^{a}\right)  \right\vert ^{2}ds\right\vert
^{p/2}\right] \\
&  \leq\frac{C}{N^{p}}E\left[  \left\vert \int_{0}^{t}\sum_{a\in\Lambda
_{s}^{N}}\left[  e^{\left(  t-s\right)  A}\nabla\theta_{N}\right]  _{\alpha
}^{2}\left\vert x-x^{\prime}\right\vert ^{2\alpha}ds\right\vert ^{p/2}\right]
\\
&  =\frac{C}{N^{p}}E\left[  \left\vert \int_{0}^{t}N\left[  S_{s}^{N}\right]
\left[  e^{\left(  t-s\right)  A}\nabla\theta_{N}\right]  _{\alpha}%
^{2}ds\right\vert ^{p/2}\right]  \left\vert x-x^{\prime}\right\vert ^{\alpha
p}\\
&  \leq CE\left[  \left[  S_{T}^{N}\right]  ^{p/2}\right]  \left(  \frac{1}%
{N}\int_{0}^{t}\left[  e^{\left(  t-s\right)  A}\nabla\theta_{N}\right]
_{\alpha}^{2}ds\right)  ^{p/2}\left\vert x-x^{\prime}\right\vert ^{\alpha p}.
\end{align*}
To apply Kolmogorov criterion we need $\alpha p>d$. If we choose $\alpha>0$
such that
\[
\frac{1}{N}\int_{0}^{t}\left[  e^{\left(  t-s\right)  A}\nabla\theta
_{N}\right]  _{\alpha}^{2}ds\leq C,
\]
then we can take $p$ so large that $\alpha p>d$. Hence, we may choose $\alpha$
as small as we want. We denote the uniform norm in the space of continuous
functions by $\left\Vert \cdot\right\Vert _{0}$. A rough computation (we do
not give details)\ gives us%
\[
\left[  e^{\left(  t-s\right)  A}\theta_{N}\right]  _{\alpha}^{2}\leq\left(
\frac{C}{\left(  t-s\right)  ^{\frac{\alpha}{2}+\frac{1}{2}-\alpha}}\right)
^{2}\left\Vert \left(  1-A\right)  ^{-\frac{1}{2}+\alpha}\nabla\theta
_{N}\right\Vert _{0}^{2}%
\]
hence%
\[
\frac{1}{N}\int_{0}^{t}\left[  e^{\left(  t-s\right)  A}\nabla\theta
_{N}\right]  _{\alpha}^{2}ds\leq\frac{C}{N}\left\Vert \left(  1-A\right)
^{-\frac{1}{2}+\alpha}\nabla\theta_{N}\right\Vert _{0}^{2}\leq\frac{C}%
{N}\left\Vert \left(  1-A\right)  ^{\alpha}\theta_{N}\right\Vert _{0}^{2}%
\leq\frac{C}{N}N^{2\beta+2\alpha\beta/d}.
\]
Since $\alpha$ can be taken arbitrarily small, the condition is $\beta<1/2$.

\begin{acknowledgement}
The work of F. F. is supported in part by University of Pisa under the Project
PRA\_2016\_41. The work of M. L. is supported in part by DFG RTG 1845.The work
of C. O. is supported in part by FAPESP 2015/04723-2 and CNPq through the
grant 460713/2014-0.
\end{acknowledgement}

\end{document}